\documentclass{article}
\usepackage{graphicx, amsmath, amssymb,mathrsfs,amsthm,lineno}

\usepackage{color}

\theoremstyle{definition}
\newtheorem{thm}{Theorem}[section]
\newtheorem{lem}[thm]{Lemma}

\newtheorem{prop}[thm]{Proposition}
\newtheorem{defn}[thm]{Defninition}
\newtheorem{rem}[thm]{Remark}

\newtheorem{note}[thm]{Note}
\makeatletter

\@addtoreset{equation}{section}
\makeatother

\title{Geometric mean of probability measures and geodesics of Fisher information metric}

\author{Mitsuhiro Itoh\footnote{Institute of Mathematics, University of Tsukuba,
1-1-1 Tennodai, Tsukuba-shi, Ibaraki 305-8577, JAPAN\quad e-mail : itohm@math.tsukuba.ac.jp} and Hiroyasu Satoh\footnote{Liberal Arts and Sciences, Nippon Institute of Technology, 4-1 Gakuendai, Miyashiro-machi, Minamisaitama-gun, Saitama 345-8501 JAPAN\quad e-mail : hiroyasu@nit.ac.jp}}
\date{\today}
\begin{document}
\maketitle

\begin{abstract}
The space of all probability measures having positive density function on a connected compact smooth manifold $M$, denoted by $\mathcal{P}(M)$, carries the Fisher information metric $G$.
We define the geometric mean of probability measures by the aid of which we investigate information geometry of $\mathcal{P}(M)$, equipped with $G$.
We show that a geodesic segment joining arbitrary probability measures $\mu_1$ and $\mu_2$ is expressed by using the normalized geometric mean of its endpoints.
As an application, we show that any two points of $\mathcal{P}(M)$ can be joined by a unique geodesic.
Moreover, we prove that the function $\ell$ defined by $\ell(\mu_1, \mu_2):=2\arccos\int_M \sqrt{p_1\,p_2}\,d\lambda$, $\mu_i=p_i\,\lambda$, $i=1,2$ gives the Riemannian distance function on $\mathcal{P}(M)$.
It is shown that geodesics are all minimal.
\end{abstract}

\section{Introduction}

For positive numbers $a$ and $b$, $\sqrt{a\,b}$ is called the geometric mean of $a$ and $b$.
The geometric mean of probability measures is similarly defined as follows; for two probability measures of density functions $p_1$ and $p_2$, we define their geometric mean by $\sqrt{p_1\,p_2}$.
By normalizing it, we obtain a probability measure.

In this paper we study, from a viewpoint of the normalized geometric mean, information geometry of the space $\mathcal{P}(M)$ of probability measures on a manifold $M$, which is equipped with Fisher information metric $G$.
By the aid of the normalized geometric mean,
we give a formula describing geodesic segments and then exhibit an exact form of the distance function for the space of probability measures with respect to the metric $G$.

Let $M$ be a connected, compact smooth manifold with a smooth probability measure $\lambda$.
Let $\mathcal{P}(M)$ be the space of probability measures on $M$ which are absolutely continuous with respect to the measure $\lambda$ and have positive continuous density function:
\begin{equation}
\mathcal{P}(M) = \left\{ \mu \left\vert\, \mu\, {\rm is\, a\, measure\, on}\, M, \int_M d\mu=1, \mu \ll \lambda, \frac{d\mu}{d\lambda}\in C^0_+(M) \right.\right\}.
\end{equation}
Here $d\mu/d\lambda$ is the Radon-Nikodym derivative of $\mu$ with respect to $\lambda$ and $C^0_+(M)$ denotes the set of all positive continuous functions on $M$.
The geometric mean of $\mu_1= p_1\, \lambda$, $\mu_2= p_2\, \lambda \in \mathcal{P}(M)$ is defined by $\sqrt{p_1\,p_2}\,\lambda$.
By normalizing the geometric mean, we give the definition of the normalized geometric mean.

\begin{defn}\label{normalizedgeometricmean}
The normalized geometric mean is a map $\varphi : \mathcal{P}(M)\times \mathcal{P}(M) \rightarrow \mathcal{P}(M)$ defined by
\begin{equation}
\varphi(\mu_1,\mu_2) = \left(\int_{x\in M} \sqrt{\frac{d\mu_2}{d\mu_1}(x)}\ d\mu_1(x) \right)^{-1}\sqrt{\frac{d\mu_2}{d\mu_1}}\, \mu_1.
\end{equation}
\end{defn}

We remark that $\displaystyle{\sqrt{\frac{d\mu_2}{d\mu_1}}\, \mu_1 = \sqrt{p_1\ p_2}\hspace{0.5mm}\lambda }$ for $\mu_i= p_i\,\lambda$, $i=1,2$ and then $\varphi(\mu_1,\mu_2) = \varphi(\mu_2,\mu_1)$, $\varphi(\mu,\mu)=\mu$.

\begin{defn}\label{ellfunction}
Let $\ell : \mathcal{P}(M)\times \mathcal{P}(M) \rightarrow [0,\pi)$ be a function defined by
\begin{equation}\label{distancefunction}
\ell(\mu_1,\mu_2)
= 2\, {\rm arccos}\left(\int_{x\in M}\sqrt{\frac{d\mu_2}{d\mu_1}(x)}\ d\mu_1(x) \right).
\end{equation}
\end{defn}

The aim of this paper is to present geometric characterization of the map $\varphi$ and the function $\ell$ from information geometry of $\mathcal{P}(M)$.

We mention here the informations which are closely related to $\varphi$ and $\ell$.
The integration
$$
C_H(\mu_1, \mu_2):
=\int_{x\in M} \sqrt{\frac{d\mu_1}{d\lambda}(x)}\sqrt{\frac{d\mu_2}{d\lambda}(x)}\,d\lambda(x)
=\int_{x\in M} \sqrt{\frac{d\mu_2}{d\mu_1}(x)}\,d\mu_1(x)
$$
is called the \textit{Hellinger integral} or the \textit{Hellinger coefficient}, representing the amount that measures the separation of two probability measures.
The function $\ell$ defined at Definition  \ref{ellfunction} is then expressed as $\ell(\mu_1,\mu_2)=2\arccos C_H(\mu_1, \mu_2)$. 

The information given by
$$
d_H(\mu_1, \mu_2):=\left\{\int_M\left(\sqrt{\frac{d\mu_1}{d\lambda}}-\sqrt{\frac{d\mu_1}{d\lambda}} \right)^2d\lambda \right\}^{1/2}=2(1-C_H(\mu_1, \mu_2)),
$$
called the Hellinger distance \cite{Nik-EOM},
is characterized as the square of the $0$-divergence (see \cite[p.58]{Amari2016}).

The function $\ell$ provides a Riemannian distance function with respect to a certain Riemannian metric, Fisher information metric, as stated in Theorem \ref{maincor}.

We regard the space $\mathcal{P}(M)$ as an infinite dimensional manifold whose tangent space $T_{\mu}\mathcal{P}(M)$ at $\mu\in \mathcal{P}(M)$ is identified with the vector space
\begin{equation}
\left\{\tau\, \left\vert\, \mbox{$\tau$ is a signed measure on $M$},\, \int_M d\tau = 0, \frac{d\tau}{d\mu}\in C^0(M)\right.\right\}.
\end{equation}
T. Friedrich \cite{F1991} defines for each $\mu\in\mathcal{P}(M)$ an inner product $G_{\mu}$ of $\tau_1, \tau_2\in T_{\mu}\mathcal{P}(M)$ by
\begin{equation}\label{FisherIM}
G_{\mu}(\tau_1,\tau_2) = \int_M\ \frac{d\tau_1}{d\mu}\, \frac{d\tau_2}{d\mu}\,d\mu
\end{equation}
which is a natural extension of the Fisher information matrix for a statistical model in mathematical statistics and information theory (see \cite{AN2000}).
We call the map $\mu \mapsto G_{\mu}$ the {\it Fisher information metric} on $\mathcal{P}(M)$.
The metric $G$ is invariant under the push-forward transformation of probability measures as easily observed (see \cite[Satz 1]{F1991}).
 Namely, any homeomorphism of $M$ is an isometry with respect to the metric $G$ via the push-forward transformation of $\mathcal{P}(M)$.
Remark that the group of homeomorphisms of a compact manifold $M$ acts on $\mathcal{P}(M)$ transitively via the push-forward, that is, for any $\mu\in\mathcal{P}(M)$ there exists a homeomorphism $\Phi$ of $M$ such that $\Phi_{\sharp}\lambda= \mu$. Here $\Phi_{\sharp}$ means the push-forward.
Refer for this to \cite{Fa, OxUl}.
This fact tells us that the space $\mathcal{P}(M)$ consisting of probability measures of continuous density function admits a structure of a Riemannian homogeneous space.
Refer to \cite{Baueretal} for the uniqueness of the Fisher metric on the space of probability measures having smooth density function   under push-forward invariance of diffeomorphisms. Notice the space of probability measures of smooth density function is a dense subset of $\mathcal{P}(M)$.

An embedding $\rho : \mathcal{P}(M) \rightarrow L_2(M,\lambda);\, \mu = p \lambda \mapsto \sqrt{p}$ provides the space $\mathcal{P}(M)$ an $L_2$-topology.
Here $L_2(M,\lambda)$ is the $L_2$-space of integrable functions on $M$ of finite norm $\|\cdot\|_{L_2}$, where the norm is defined by $\| f\|_{L_2} = \left(\int_M \vert f\vert^2 d \lambda\right)^{1/2}$. Then, $\mathcal{P}(M)$ is embedded onto the subset $\rho(\mathcal{P}(M)) \subset \{f\in L_2(M,\lambda)\,|\, \| f\|_{L_2}=1\}$ of $L_2(M,\lambda)$.
We equip each $\mu=p \lambda\in \mathcal{P}(M)$ with an $\varepsilon$-neighborhood of $\mu$ in the $\|\sqrt{\cdot}-\sqrt{\cdot}\|_{L_2}$-topology as $\{ \mu'= p' \lambda\in \mathcal{P}(M) \, | \, \| \sqrt{p'} - \sqrt{p}\,\|_{L_2} < \varepsilon\}$ for $\varepsilon >0$.
Notice that $\mathcal{P}(M)$ admits also the $C^0$-topology with the norm $\| p\|_{C^0} := {\rm sup}_{x\in M} \vert p(x)\vert$ for $\mu = p \lambda$.
However, in this paper we employ mainly the $\|\sqrt{\cdot}-\sqrt{\cdot}\|_{L_2}$--topology.
The map $\varphi$  and the function $\ell$ are  continuous with respect to the product topology of $\mathcal{P}(M)\times \mathcal{P}(M)$ induced from $\|\sqrt{\cdot}-\sqrt{\cdot}\|_{L_2}$-topology, as shown in section \ref{continuity}.
See section \ref{relevanttopology} for an appropriate smooth structure on $\mathcal{P}(M)$, given in \cite{Pistone}.
The tangent space $T_{\mu}\mathcal{P}(M)$ is an infinite dimensional vector space with the inner product $G_{\mu}$. The vector space $T_{\mu}\mathcal{P}(M)$ is not a Hilbert space, since the completion of the space $C^0(M)$ is not itself $C^0(M)$ so that $\mathcal{P}(M)$ is not a Riemannian-Hilbert manifold.
Remark that the pullback of the $L_2$-inner product $(\cdot,\cdot)_{L_2}$, given by $(f, f_1)_{L_2}=\int_{x\in M} f(x) f_1(x) d\lambda(x)$,
via $\rho$ coincides with $\frac{1}{4}G(\cdot,\cdot)$.

\begin{rem}
The compactness of the manifold $M$ is assumed throughout this paper. When $M$ is non-compact,  the argument appeared in this paper is almost valid if a minor change is done,  as that $\mathcal{P}(M)$ is the space of all probability measures $\mu=p(x)\lambda$, $\mu \ll \lambda$ such that $\mu$  is connected with $\lambda$ by an open mixture arc (for the notion of open mixture arc see subsection \ref{openmixturearc} and \cite{CPistone, Santacrocexx}) with $p=p(x)\in C^0_+(M)$.
Then the $\|\sqrt{\cdot}-\sqrt{\cdot}\|_{L_2}$--topology is introduced on $\mathcal{P}(M)$, same as in the compact manifold case.
The tangent space $T_{\mu}\mathcal{P}(M)$ is the vector space of measures $\tau= q(x)\lambda$ of $q\in C^0(M)$ such that  there exists an $\varepsilon>0$ for which $\mu+ t\tau=(p+t q)\lambda$ defines a probability measure in $\mathcal{P}(M)$ for any $t\in (-\varepsilon,\varepsilon)$.

Sections \ref{geodesic} and \ref{continuity} may be valid even for a non-compact manifold $M$. 
We will give in future a relevant study about non-compact manifold case.
\end{rem}

Let $\nabla$ be the Levi-Civita connection of the metric $G$.
Then $\nabla$ is given by
\begin{equation}\label{LCconnFI}
\nabla_{\tau_1}\tau_2(\mu)
=-\frac{1}{2}\left(
\frac{d\tau_1}{d\mu}\frac{d\tau_2}{d\mu}-\int_M\frac{d\tau_1}{d\mu}\frac{d\tau_2}{d\mu}d\mu
\right)\mu
\end{equation}
for any $\tau_1, \tau_2\in T_{\mu}\mathcal{P}(M)$ (see \cite[p.276]{F1991}).
T. Friedrich computes the Riemannian curvature tensor of $G$ by using \eqref{LCconnFI} and shows that the space $\mathcal{P}(M)$ equipped with the metric $G$ has constant sectional curvature $+1/4$ (\cite[Satz 2]{F1991}).
He also obtains an explicit formula for a curve in $\mathcal{P}(M)$ to be geodesic with respect to $G$ for a given initial data. 
In fact, let $\gamma : I \rightarrow \mathcal{P}(M)$($I\subset \mathbb{R}$ is an open interval, $0\in I$) be a geodesic, parametrized by arc-length with an initial data: $\gamma(0)= p_0 \lambda$, $\dot{\gamma}(0) = \dot{p}_0\lambda$ of $\vert \dot{\gamma}(0)\vert_{\mu} = 1$.
Then the density function $p_t = p_t(x)$ of $\gamma(t)$ with respect to $\lambda$ has the form
\begin{equation}\label{formulageodesic}
p_t(x) = \frac{1}{1+\tan^2(\frac{t}{2})}\left\{p_0(x)+ 2 \tan \left(\frac{t}{2}\right) \dot{p}_0(x) + \tan^2 \left(\frac{t}{2}\right)\, \frac{\dot{p}_0^2(x)}{p_0(x)}\right\}.
\end{equation}
From this formula any geodesic of $\mathcal{P}(M)$ is seen to be periodic with period $2\pi$. It is true that $\gamma(t) = p_t \lambda$ is indeed a probability measure for any $t$.
However, it is not determined from \eqref{formulageodesic} whether $\gamma(t)= p_t \lambda$ belongs to $\mathcal{P}(M)$.
It is also not mentioned in \cite{F1991} whether $p_t\in C^0_+(M)$ at any $t$ for which $\gamma(t)$ is defined. However, this is completely solved for a geodesic segment, by the aids of the density free expression for geodesic together with the notion of normalized geometric mean.
 
Every geodesic is incomplete, as we see from \eqref{formulageodesic} $\gamma(\pm \pi) \not\in \mathcal{P}(M)$, because $\gamma(t)$ at $t=\pm\pi$ has the form $\displaystyle \gamma(\pm\pi) = \left(\frac{\dot{p}_0(x)}{p_0(x)} \right)^2\ \lambda$ and $\displaystyle \int_M \dot{p}_0(x)\,d\lambda(x) = 0 $ so that the continuous function $\dot{p}_0(x)$ admits necessarily a zero in $M$.
It is of interest whether an interval $I \subset (-\pi,\pi)$, on which the geodesic $\gamma$ is defined, can be extended to a maximal one.
 
We emphasize that by relaxing the continuity of density function for probability measures, the situation for geodesics is drastically changed, as will be seen in Proposition \ref{ellfinite}, for example, uniqueness of geodesic segment for given endpoints collapses.

In \cite{IS2015} we obtain from \eqref{formulageodesic} a density free description of a geodesic in $\mathcal{P}(M)$ by the aid of which we derive an explicit formula representing a geodesic segment $\gamma(t)$ for given two endpoints $\mu, \mu_1\in\mathcal{P}(M)$.
By using the normalized geometric mean, we obtain the following theorem stating uniqueness and existence of geodesic segment.

\begin{thm}\label{main1}
Let $\mu, \mu_1\in\mathcal{P}(M)$ be arbitrary distinct probability measures.
Then, there exists a unique geodesic $\gamma(t)$ with respect to $G$ parametrized by arc-length, joining $\mu$ and $\mu_1$, and being expressed in the form
\begin{equation}\label{expression-g}
\gamma(t)=a_1(t)\,\mu+a_2(t)\,\mu_1+a_3(t)\,\varphi(\mu, \mu_1),
\qquad t\in [0, l].
\end{equation}
Here $\gamma(l)=\mu_1$, $l=\ell(\mu, \mu_1)$ and $a_i(t)$, $i=1,2,3$ are the non-negative functions of $t$ satisfying
\begin{equation*}
a_1(t)+a_2(t)+a_3(t)=1,
\end{equation*}
which are given by
\begin{equation*}\label{geodesic3}
a_1(t)=\left(\frac{\sin(l-t)/2}{\sin(l/2)}\right)^2,\quad
a_2(t)=\left(\frac{\sin(t/2)}{\sin(l/2)}\right)^2,\\
\end{equation*}
\begin{equation*}\label{geodesic3}
a_3(t)=\frac{2\cos(l/2)\cdot\sin(t/2)\cdot\sin(l-t)/2}{\sin^2(l/2)}.
\end{equation*}
\end{thm}

The uniqueness of a geodesic segment follows from the fact that all probability measures in $\mathcal{P}(M)$ and tangent vectors have continuous density function on $M$.

From Theorem \ref{main1},
we find the following properties of geodesics in $\mathcal{P}(M)$.
\begin{thm}\label{maincor}
Let $\gamma=\gamma(t)$, $t\in [0,l]$ be a geodesic segment joining distinct probability measures $\mu$, $\mu_1\in \mathcal{P}(M)$ such that $\gamma(0)=\mu$, $\gamma(l)=\mu_1$.
Then,
\begin{enumerate}
\item $\gamma(t)$ belongs to $\mathcal{P}(M)$ at any $t\in[0, l]$,
\item the geodesic segment $\gamma : [0,l] \rightarrow \mathcal{P}(M)$ is a curve lying on the plane spanned by $\mu, \mu_1$ and their normalized geometric mean $\varphi(\mu, \mu_1)$,
\item the velocity vectors of the geodesic segment at $t=0$ and $t= l$ are respectively given by $\dot{\gamma}(0)=\cot\frac{l}{2}\left(\varphi(\mu, \mu_1)-\mu\right)$ and
$\dot{\gamma}(l)=-\cot\frac{l}{2}\left(\varphi(\mu, \mu_1)-\mu_1\right)$. This implies that two tangent lines defined at the endpoints of the geodesic segment always intersect each other at $\varphi(\mu, \mu_1)$ (see Remark \ref{tangentline_gm}) and
\item the midpoint of the geodesic segment $\gamma(t)$, $t\in[0,l]$ is represented by
\begin{equation}
\gamma(l/2)=\frac{1}{4\cos^2(l/2)}\left(1+\sqrt{\frac{d\mu_1}{d\mu}}\right)^2\mu.
\end{equation}
The probability measure at the right hand side is viewed as the normalized $(1/2)$-power mean of endpoints $\mu, \mu_1$.
Here the normalized $\alpha$-power mean $\varphi^{(\alpha)}(\mu, \mu_1), \alpha\in\mathbb{R}$, of probability measures $\mu, \mu_1$ is defined by
\begin{equation}\label{nalphapower}
\varphi^{(\alpha)}(\mu, \mu_1)=\left[
\int_M\left\{1+\left(\frac{d\mu_1}{d\mu}\right)^\alpha\right\}^{1/\alpha}\, d \mu
\right]^{-1}
\left\{1+\left(\frac{d\mu_1}{d\mu}\right)^\alpha\right\}^{1/\alpha}\mu.
\end{equation}
\end{enumerate}
\end{thm}

The normalized $\alpha$-power mean is derived from the $\alpha$-power mean of positive two numbers $a$ and $b$ defined by $\left(\frac{a^\alpha+b^\alpha}{2}\right)^{1/\alpha}$ (see \cite{B}).
In particular, the arithmetic mean, the geometric mean and the harmonic mean are $\alpha$-power means, $\alpha= +1, 0$ and $-1$, respectively. 

\begin{rem}
A. Ohara considers in \cite{O} operator means on a symmetric cone $\Omega$ and a dualistic structure naturally introduced on it, i.e., a Riemannian metric $g$ on $\Omega$ together with affine connections $(\nabla, \nabla^\ast)$ adjoint each other with respect to $g$.
In particular, he constructs a family of affine connections $\{\nabla^{(\alpha)}\}$ such that $\nabla^{(-\alpha)}$ is the dual connection of $\nabla^{(\alpha)}$ and $\nabla^{(0)}$ is the Levi-Civita connection of $g$, and shows that the midpoint of $\nabla^{(\alpha)}$-geodesic segment is the $\alpha$-power mean of their endpoints.
Theorem \ref{maincor} (iv) is inspired by his consideration.

The reader may question the difference between Ohara's results and Theorem \ref{maincor} (iv), because Ohara asserts that the midpoint of $\nabla^{(\alpha)}$-geodesic segment is characterized as the $\alpha$-power mean and we assert that the midpoint of $\nabla^{(0)}$-geodesic segment is characterized as the $(1/2)$-power mean.
In our case, we consider the Fisher metric defined on an infinite dimensional space which is natural extension of the Fisher matrix.
On the other hand, Ohara considers the Hessian metric and $\alpha$-connections induced by a certain potential function.
In this way, the structure which we treat is different from the structure Ohara considers.
\end{rem}

We are able to define similarly $\alpha$-connections on $\mathcal{P}(M)$,
which also play a significant role in information geometry,
and obtain in a subsequent paper a certain relation between the midpoint of a geodesic segment of $\alpha$-connection and the normalized $\alpha$-power mean of their endpoints.

\begin{rem}
The authors considered in \cite{IS2015} a Hadamard manifold $X$, a simply connected, complete Riemannian manifold having non-positive curvature, and the space $\mathcal{P}(\partial X)$ of probability measures defined on the ideal boundary $\partial X$ of $X$.
Under certain assumptions, we can define a map $\mathrm{bar} : \mathcal{P}(\partial X)\rightarrow X$, called the barycenter map, as a critical point of a function $\displaystyle \mathbb{B}_\mu : X\rightarrow \mathbb{R}$ given by $\mathbb{B}_\mu(x)=\int_{\theta\in\partial X}B_\theta(x)\,d\mu(\theta)$, where $B_{\theta}(x)$ is the Busemann function associated with $\theta\in\partial X$, geometrically defined on a Hadamard manifold.
The barycenter map plays an essential role in the proof of Mostow's rigidity theorem shown by G. Besson et al. \cite{BCG}, following the idea of Douady and Earle \cite{DouEa}.
In \cite[Theorem 5]{IS2015}, the authors show that the map $\mathrm{bar} : \mathcal{P}(\partial X)\rightarrow X$ is an onto fibration and then investigate certain conditions for a geodesic segment of $\mathcal{P}(\partial X)$ under which the endpoints of the geodesic segment are contained in a common fiber $\mathrm{bar}^{-1}(x)$, $x\in M$. For other directions of geometry of $\mathcal{P}(\partial X)$ with respect to Fisher information metric refer to \cite{IShi2008, ISShi2008, IS2010, IS2017j}.
\end{rem}

The following theorem indicates that the function $\ell$, defined in \eqref{distancefunction} is actually the Riemannian distance function of the space $\mathcal{P}(M)$.

\begin{thm}\label{main2}
$\ell(\mu, \mu_1)$ gives the Riemannian distance between $\mu$ and $\mu_1$ with respect to the Fisher information metric $G$.
\end{thm}

This theorem is verified by the aid of three propositions, familiar in a finite dimensional Riemannian geometry; Gauss lemma, the existence theorem of totally normal neighborhood and the minimizing length properties of geodesics, cf. \cite[Chap. 3]{doC}.

\begin{rem}
T. Friedrich also stated that $\ell(\mu_1, \mu_2)$ is the Riemannian distance between $\mu_1$ and $\mu_2$, but without a proof (see \cite[p.279, Bemerkung]{F1991}).
\end{rem}

From Theorem \ref{main2}, the Riemannian distance satisfies $\ell(\mu_1,\mu_2) < \pi$ for all $\mu_1,\mu_2\in \mathcal{P}(M)$.
Therefore the diameter ${\mathcal D}$ of $\mathcal{P}(M)$ with respect to the metric $G$ fulfills ${\mathcal D} \leq \pi$.
The diameter is here defined by
\begin{equation*}
\mathcal{D} = \sup\,\big\{\ell(\mu_1,\mu_2)\, \vert\, \mu_1,\,\mu_2\in \mathcal{P}(M)\big\}.
\end{equation*}

\begin{thm}
The diameter ${\mathcal D}$ of $\mathcal{P}(M)$ with respect to the metric $G$ satisfies ${\mathcal D} = \pi$. 
\end{thm}

This theorem can be verified, by applying the parametrix of the heat kernel of a compact smooth Riemannian manifold $M$.
For the details, refer to \cite{IS2020}.
 
Now we will briefly state the development of information geometry and its topics related to this paper.
Information geometry which is the geometry on the space of probability distributions, called the statistical model, began with the geometrical considerations of statistical estimations.
C. R. Rao \cite{Rao} proposed defining a metric based on the Fisher matrix and S. Amari gave a modern differential geometric framework, i.e., a Riemannian metric and affine connections, on his idea (see \cite{AN2000}).
Although information geometry developed afterwards, the subject was only a family of probability distributions whose parameter space has finite dimension.
Since the 1990s, the information geometry of the infinite dimensional case, i.e., the geometric structure on the space of all probability distribution has begun to be considered.
In 1991, T. Friedrich extended the Fisher metric on infinite dimensional statistical model and investigated properties of Riemannian geometric nature, for example the Riemannian curvature tensor and geodesics, and symplectic structures without any argument of the coordinate structure of the space of probability measures.
In 1995, G. Pistone and C. Sempi \cite{PistoneS} defined the topology of the space of all positive densities of the probability measures, which is a subset of $L_1$-space, as a Banach manifold whose model space is the Orlicz space.
The geometrical and analytical properties of the mixture model $\mathcal{M}(\mu)$ and the exponential model $\mathcal{E}(\mu)$ have been studied by Pistone and his coauthors (for example, see \cite{Pistone,Pistone-2,CPistone,GPistone}).
See also \cite{Santacrocexx}.

Our argument is based on Friedrich's framework.
We can develop information geometry for a more general setting of probability spaces by the aids of the researches of Pistone-Sempi (for their study refer to \cite{PistoneS} and \cite{CPistone}).
In final section we will outline their argument by means of Orlicz spaces.
We show further in Proposition \ref{localchartform} that Fisher information metric $G$, given at \eqref{FisherIM}, can be represented as the covariance of random variables in a local chart representation, by the framework of Pistone-Sempi.

This paper is organized as follows.
In section \ref{geodesic}, we outline the derivation of a geodesic $\gamma(t)$ for a given initial data $\gamma(0)=\mu$ and $\dot{\gamma}(t)=\tau$, and for a boundary data $\gamma(0)=\mu$ and $\gamma(l)=\mu_1$, respectively.
Moreover, we show Theorems \ref{main1} and Theorem \ref{maincor}, which state a geometric characterization of the normalized geometric mean in Fisher information geometry.
Section \ref{continuity} is devoted to showing that $\varphi$ and $\ell$ are continuous  with respect to the $\|\sqrt{\cdot}-\sqrt{\cdot}\|_{L_2}$-topology.
In section \ref{exponential}, we consider the exponential map and a totally normal neighborhood on $\mathcal{P}(M)$ and  verify Theorem \ref{main2}.
In final section, we consider the topology and the smooth structure of $\mathcal{P}(M)$.
The argument of Pistone and Sempi is summarized and the notion of being connected by an open mixture arc together with Proposition \ref{definitiontangentspace} concerning with constant vector field argument is given.


\section{Geodesics with respect to Fisher information metric}\label{geodesic}

\subsection{Initial value problem}

We outline the derivation of a formula of geodesic in $\mathcal{P}(M)$ by following the argument of T. Friedrich (see \cite[\S 2]{F1991} for details).

Let $\lambda\in\mathcal{P}(M)$ be the probability measure represented by the Riemannian volume form of $M$, associated with a Riemannian metric, provided $M$ is orientable.
For non-orientable $M$ choose the double covering $\tilde{M}$ of $M$ and then taking the push-forward of the Riemannian volume form $\lambda_{\tilde{M}}$ via the double covering map $\pi : \tilde{M} \rightarrow M$.

Denote by $\gamma(t)=p_t\,\lambda$ a geodesic in $\mathcal{P}(M)$ which is parametrized by arc-length, and whose initial point is $\gamma(0)=\mu$ and initial unit velocity is $\tau\in T_\mu\mathcal{P}(M)$.
Here $p_t : x\mapsto p_t(x)$ is a continuous function on $M$ which is assumed to be $C^1$-class with respect to $t$.
Since $\gamma(t)$ is a geodesic, we have
\begin{equation*}
G(\nabla_{\dot{\gamma}(t)}\dot{\gamma}(t), \tau)
=\dot{\gamma}(t)\,G(\dot{\gamma}(t), \tau)
-G(\dot{\gamma}(t), \nabla_{\dot{\gamma(t)}}\tau)=0
\end{equation*}
for any constant vector field $\tau$.
Then, by using the formula \eqref{LCconnFI} for the Levi-Civita connection with respect to $G$, we find that $p_t$ satisfies
\begin{equation}
\frac{d}{dt}\left(\frac{\dot{p}_t}{p_t}\right)+\frac{1}{2}\left(\frac{\dot{p}_t}{p_t}\right)^2+\frac{1}{2}=0.
\end{equation}
Setting $f_t=\dot{p}_t/p_t$, we obtain $\dot{f}_t+\frac{1}{2}{f_t}^2+\frac{1}{2}=0$
and find that a solution to this differential equation is $f_t=\tan\left(-1/2+A\right)$.
Hence we have 

$\log p_t=2\log \cos\left(-t/2+A\right) + B$, i.e.,
\begin{equation*}
p_t=C\,\cos^2\left(-\frac{t}{2}+A\right),\qquad C=\exp B
\end{equation*}
where $A$ and $C$ are functions on $M$ determined by the initial condition as follows:
\begin{equation*}
A=\arctan\left(\frac{\dot{p}_0}{p_0}\right),\qquad C=\frac{(p_0)^2+(\dot{p}_0)^2}{p_0}.
\end{equation*}
Hence we have the following.

\begin{prop}[\cite{F1991}]\label{densitygeodesic}
\begin{align*}
p_t=&\frac{(p_0)^2+(\dot{p}_0)^2}{p_0}
\cos^2\left(-\frac{t}{2}+\arctan\left(\frac{\dot{p}_0}{p_0}\right)\right)\\
=&\frac{1}{1+\tan^2(t/2)}\left\{
p_0+2\,\dot{p}_0\tan\frac{t}{2}+\frac{(\dot{p}_0)^2}{p_0}\cdot\tan^2\frac{t}{2}
\right\}.
\end{align*}
\end{prop}

The following is the density free expression of a geodesic.
\begin{prop}\label{densityfree}
Let $\gamma(t)$ be a geodesic with $\gamma(0)=\mu$ and $\dot{\gamma}(0)= \tau$. If $\tau$ is of unit norm, i.e., $\vert \tau\vert_{\mu}=1$ with respect to $G$, then $\gamma(t)$ is represented by
\begin{equation}\label{Fgeodesic}
\gamma(t)=\left(\cos\frac{t}{2}+\frac{d\tau}{d\mu}\cdot\sin\frac{t}{2}\right)^2\mu.
\end{equation}
\end{prop}

In fact, set $\mu=p_0\,\lambda$, $\tau=\dot{p}_0\,\lambda$ and obtain from Proposition \ref{densitygeodesic}
\begin{align*}
\gamma(t)=p_t\,\lambda
=&\frac{1}{1+\tan^2(t/2)}\left\{
1+2\,\frac{\dot{p}_0}{p_0}\cdot\tan\frac{t}{2}+\left(\frac{\dot{p}_0}{p_0}\right)^2\tan^2\frac{t}{2}
\right\}\mu\\
=&\frac{1}{1+\tan^2(t/2)}
\left(1+\frac{\dot{p}_0}{p_0}\cdot\tan\frac{t}{2}\right)^2\mu\\
=&\left(\cos\frac{t}{2}+\frac{d\tau}{d\mu}\cdot\sin\frac{t}{2}\right)^2\mu.
\end{align*}

\begin{rem}\label{rem_para}
We notice from \eqref{densityfree} that 
$\displaystyle \gamma(\pm\pi)=\left(d\tau/d\mu\right)^2\mu$ is a probability measure. However, it does not admit positive density function, as we remarked in section 1.
Moreover the formula \eqref{Fgeodesic} indicates that every geodesic is periodic with period $2\pi$, since
\begin{equation*}
\gamma(t) = \left\{\frac{1}{2}(1+\cos t) + \frac{1}{2}(1-\cos t) \left(\frac{d\tau}{d\mu}\right)^2\right\} \mu + \sin t\,\tau.
\end{equation*}
Therefore we are able to choose a parameter $t$, at which $\gamma(t)$ is defined, is inside the open interval $(-\pi,\pi)$. 
\end{rem}

\subsection{Boundary value problem}

Next, we rewrite \eqref{densityfree} by using the boundary data (see \cite[Theorem 11]{IS2015}).

\begin{thm}\label{uniquenessthm}
Let $\mu$, $\mu_1$ be arbitrary probability measures of $\mathcal{P}(M)$.
Assume $\mu\not=\mu_1$. Then there exists a unique geodesic segment $\gamma(t)$, $t\in [0,l]$, $l = \ell(\mu,\mu_1)$ such that $\gamma(0)=\mu$, $\gamma(l)= \mu_1$.
In fact, $\gamma(t)$ is represented as
\begin{equation*}
\gamma(t) = \left(\cos\frac{t}{2}+\frac{d\tau}{d\mu}\cdot\sin\frac{t}{2}\right)^2\mu
\end{equation*}
with initial velocity vector 
\begin{equation*}
\tau =\frac{1}{\sin(l/2)}\left(
\sqrt{\frac{d\mu_1}{d\mu}}-\cos\frac{l}{2}
\right)\mu.
\end{equation*}
\end{thm}

\begin{proof}
If we assume that $\mu$ and $\mu_1$ are joined by \eqref{Fgeodesic}, then there exists a positive number $l$ such that $\gamma(l)=\mu_1$, i.e., it holds
\begin{equation}\label{gamma_endpoint}
\left(\cos\frac{l}{2}+\frac{d\tau}{d\mu}\cdot\sin\frac{l}{2}\right)^2\mu=\mu_1.
\end{equation}
Solving this equation with respect to $d\tau/d\mu$, by using an analogous argument in \cite[p.1830, Assertion 3]{IS2015},
we find that the initial velocity $\tau$ is uniquely determined by 
\begin{equation}\label{Fgeodesic_iv}
\tau =\frac{1}{\sin(l/2)}\left(
\sqrt{\frac{d\mu_1}{d\mu}}-\cos\frac{l}{2}
\right)\mu
\end{equation}
as follows.
In fact, from \eqref{gamma_endpoint} we have
\begin{equation*}
\left(\cos\frac{l}{2}+\frac{d\tau}{d\mu}\cdot\sin\frac{l}{2}\right)^2=\frac{d\mu_1}{d\mu},
\end{equation*}
so
\begin{equation*}
\cos\frac{l}{2}+\frac{d\tau}{d\mu}\cdot\sin\frac{l}{2}= \pm \sqrt{\frac{d\mu_1}{d\mu}}.
\end{equation*}
Define subsets $M_1$, $M_2$ of $M$ respectively by
\begin{align*}
M_1 =& \left\{ x\in M\, \left\vert\, \sin\frac{l}{2} \frac{d\tau}{d\mu}(x) = \left(-\cos\frac{l}{2}+\sqrt{\frac{d\mu_1}{d\mu}(x)}\right)\right.\right\},\\
M_2 =& \left\{ x\in M\, \left\vert\, \sin\frac{l}{2} \frac{d\tau}{d\mu}(x) = \left(-\cos\frac{l}{2}-\sqrt{\frac{d\mu_1}{d\mu}(x)}\right)\right.\right\}.
\end{align*}
The subsets $M_1, M_2$ satisfy $M_1\cup M_2=M$ and both are closed, since on a manifold $M$ the function $d\tau/d\mu$ must be continuous and the function at right hand side is also continuous.
First, we have $M_1\cap M_2=\emptyset$. This is because, if there exists, otherwise, $x\in M_1\cap M_2$, then $d\mu_1/d\mu (x)=0$ which is a contradiction. Thus, $M_1$ and $M_2$ turn out to be open and closed.
Next, we claim that $M_2=\emptyset$. 
If $M_2\not=\emptyset$, then $M=M_2$ ( and hence $M_1=\emptyset$), since $M$ is connected, and hence from $\int_M d\tau=0$ we have
\begin{equation*}
\cos\frac{l}{2}=-\int_M \sqrt{\frac{d\mu_1}{d\mu}}\,d\mu <0,
\end{equation*}
so $\pi<l<2\pi$ which is a contradiction, because $l\in(-\pi, \pi)$ (see Remark \ref{rem_para}).
Hence we have \eqref{Fgeodesic_iv} and from $\int_M d\tau=0$
\begin{equation*}
\int_M \sqrt{\frac{d\mu_1}{d\mu}}\,d\mu=\cos\frac{l}{2}.
\end{equation*}
From this and by using the normalized geometric mean $\varphi$, for the given $\mu$, $\mu_1$
we can express \eqref{Fgeodesic_iv} as
\begin{equation}\label{Fgeodesic_iv-2}
\tau =\frac{1}{\tan(l/2)}\left(\varphi(\mu, \mu_1)-\mu\right).
\end{equation}
We also have
\begin{equation}\label{Fgeodesic_l}
l=2\arccos\left(\int_M \sqrt{\frac{d\mu_1}{d\mu}}\,d\mu\right)=\ell(\mu, \mu_1).
\end{equation}
Thus the theorem is proved.
\end{proof}

If we relax the space $\mathcal{P}(M)$ of probability measures having continuous density function as the space $\tilde{\mathcal{P}}_{(L_2,\lambda)}(M)$ consisting of probability measures $\mu= p \lambda$ having $L_2$-integrable, non-negative density function $p$,
then we have the following.
 
\begin{prop}\label{ellfinite}
For given distinct $\mu, \mu_1\in \mathcal{P}(M)$ there exists a geodesic segment $\tilde{\gamma}(t)$ which joins $\mu$ and $\mu_1$, while, at least $\tilde{\gamma}(t)$ belongs to $\tilde{\mathcal{P}}_{(L_2,\lambda)}(M)$ for each $t$ such that the initial velocity vector $\dot{\tilde{\gamma}}(0)$ has $L_2$-integrable density function, but not continuous.
Furthermore $\tilde{\gamma}(t)$ satisfies $\tilde{\gamma}(0) = \mu$, $\tilde{\gamma}(\pi) = \mu_1$.
\end{prop}

\begin{proof}
We set $\mu= p \lambda$ and $\mu_1= p_1 \lambda$ with normalized geometric mean $\varphi(\mu,\mu_1)$ and set $\ell = \ell(\mu,\mu_1)$.
Here $p$, $p_1\in C^0_+(M)$.
Consider the geometric mean of $\mu$ and $\mu_1$, $\displaystyle \cos \frac{\ell}{2}\, \varphi(\mu,\mu_1)$, which is a measure given by $\sqrt{p(x) p_1(x)}\ \lambda$. 
Let $q_0(x)$ be the density function of $\varphi(\mu,\mu_1)$ with respect to $\lambda$, a positive continuous function on $M$.
Thus, $\displaystyle \cos\frac{\ell}{2} q_0(x)=\sqrt{p(x) p_1(x)}$.
We have $\displaystyle \int_M q_0(x) d\lambda =1$.
Choose a point $x_0\in M$ and let $C_{x_0}$ be the cut locus with respect to $x_0$.
Here, $\dim C_{x_0} \leq \dim M - 1$ so $C_{x_0}$ is a measure zero set with respect to $\lambda$.
For the notion and geometrical properties of cut locus refer to \cite{Sakai}.
Via the exponential map $\exp_{x_0}$, $M \setminus C_{x_0}$ is diffeomorphic to a domain $D$ of $T_{x_0}M$.
$D$ is bounded, since $M$ is compact so that there exists $R > 0$ such that $D \subset {\it B}_0(R)$, where ${\it B}_0(R)$ is the euclidean ball of radius $R$ in $T_{x_0}M$ with respect to the euclidean metric.
Let $\sigma$ be the Lebesgue's measure on $T_{x_0}M$ and identify $\sigma$ with $((\exp_{x_0})^{-1})^{\ast}\sigma$ on $M \setminus C_{x_0}$.
Then, the measure $\lambda$ restricted to $M\setminus C_{x_0}$ is represented by $\lambda\vert_{ M\setminus C_{x_0}}= f\,\sigma\vert_D$ for a positive smooth function $f$ on $D$.
The integral $\displaystyle{\int_M d \varphi(\mu,\mu_1)}$ reduces to
\begin{align*}
\int_M d \varphi(\mu,\mu_1) 
=& \int_{M\setminus C_{x_0}} q_0(x)\,d \lambda = \int_{u\in D} q_0(\exp_{x_0}u) f(u)\,d \sigma(u) \\
=& \int_{{\it B}_0(R)} \tilde{q}_0(u) \tilde{f}(u)\,d \sigma(u) = 1,
\end{align*}
where $\tilde{q}_0$ and $\tilde{f}$ are the functions on ${\it B}_0(R)$, the natural extension of $q_0(\exp_{x_0}u)$ and $f(u)$, respectively, as $\tilde{q}_0 \equiv 0$, $f \equiv 0$ on $B_0(R)\setminus D$. 

Consider the function $h$ of $r$\, given by $\displaystyle h(r) := \int_{{\it B}_0(r)} \tilde{q}_0(u) \tilde{f}(u) d \sigma(u)$ for $0 \leq r \leq R$.
It is not hard to see that $h$ is an increasing continuous function and $h(0) = 0$ holds.
By the mean value theorem for continuous functions there exists an $r_0 > 0$ such that $h(r_0)= 1/2$.
Define $\tau_1$ by
\begin{equation*}
\tau_1(x) = \left\{\begin{array}{ll}
\tilde{q}_0(u) \lambda(x); & x = \exp_{x_0}u,\, u \in B_0(r_0),\\
- \tilde{q}_0(u) \lambda(x); & x= \exp_{x_0}u,\, u\in B_0(R)\setminus B_0(r_0).\end{array}\right.
\end{equation*}
Notice
\begin{equation*}
\int_{\exp_{x_0} B_0(R)\setminus B_0(r_0)} d \tau_1 = - \int_{\exp_{x_0} B_0(R)\setminus B_0(r_0)} \tilde{q}_0(u) d\lambda(x) = -\left(1 - \frac{1}{2}\right) = - \frac{1}{2}.
\end{equation*}
Therefore, the measure $\displaystyle \cos \frac{\ell}{2}\,\tau_1$ belongs to the tangent space at $\mu$ and is of unit norm, since 
\begin{equation*}
\int_M \cos \frac{\ell}{2}\ d \tau_1 = \cos \frac{\ell}{2}\left(\int_{\exp_{x_0}{\it B}_0(r_0)} d \tau_1 + \int_{\exp_{x_0}{\it B}_0(R)\setminus{\it B}_0(r_0)} d \tau_1 \right)= \cos \frac{\ell}{2}\left( \frac{1}{2} - \frac{1}{2} \right) = 0
\end{equation*}
and $G_{\mu}(\cos \frac{\ell}{2}\,\tau_1,\cos \frac{\ell}{2}\,\tau_1)$ is given by
\begin{align*}
\cos^2 \frac{\ell}{2}\, \int_M \left(\frac{d\tau_1}{d\mu}\right)^2 d\mu
=& \cos^2 \frac{\ell}{2}\,\int_M \left(\frac{\pm q_0(x)}{p(x)}\right)^2 p(x)d\lambda\\
=& \int_M \frac{p(x)p_1(x)}{p(x)} d\lambda = \int_M p_1(x)d\lambda = \int_M d\mu_1 = 1.
\end{align*}
Set
\begin{equation*}
\tilde{\gamma}(t) = \left( \cos \frac{t}{2} + \sin \frac{t}{2}\, \cos\frac{\ell}{2}\,\frac{d\tau_1}{d\mu}\right)^2\, \mu.
\end{equation*}
Then $\tilde{\gamma}(t)$ gives a geodesic in the space $\tilde{\mathcal{P}}_{(L_2,\lambda)}(M)$.
It satisfies $\displaystyle \tilde{\gamma}(0)=\mu$ and $\displaystyle{\tilde{\gamma}(\pi) = \cos^2\frac{\ell}{2} \left(\frac{d \tau_1}{d \mu} \right)^2 \mu =\mu_1}$.
In fact, 
\begin{equation*}
\cos^2\frac{\ell}{2} \left(\frac{d \tau_1}{d \mu} \right)^2 \mu = \left(\frac{\pm q_0(x)}{p(x)}\right)^2 p(x)\lambda = \frac{q_0(x)^2}{p(x)} \lambda = \frac{p(x)p_1(x)}{p(x)}\lambda =\mu_1.
\end{equation*}
One finds easily $\displaystyle \cos\frac{\ell}{2}\tau_1\in L_2$ with respect to $\lambda$.
Thus the proposition is verified.
\end{proof}

\begin{rem}
It is not hard 
to see that $\tilde{\gamma}(t)$ gives also a geodesic in the space $\tilde{\mathcal{P}}_{(L_1,\lambda)}(M)$ with initial tangent vector having $L_1$-integrable density function.
\end{rem}

\subsection{Proofs of Theorems \ref{main1} and \ref{maincor}}

Now we return back to our main subject.
First, we prove Theorem \ref{main1}.
Substituting \eqref{Fgeodesic_iv} into \eqref{Fgeodesic}, we have
\begin{align}
\gamma(t)
=&\left\{\cos\frac{t}{2}+\sin\frac{t}{2}\cdot\frac{1}{\sin(l/2)}\left(
\sqrt{\frac{d\mu_1}{d\mu}}-\cos\frac{l}{2}
\right)\right\}^2\mu\notag\\
=&\left\{
\frac{\cos(t/2)\cdot\sin(l/2)-\sin(t/2)\cos(l/2)}{\sin(l/2)}
+\frac{\sin(t/2)}{\sin(l/2)}\sqrt{\frac{d\mu_1}{d\mu}}
\right\}^2\mu\notag\\
=&\left\{
\frac{\sin(l-t)/2}{\sin(l/2)}
+\frac{\sin(t/2)}{\sin(l/2)}\sqrt{\frac{d\mu_1}{d\mu}}
\right\}^2\mu\notag\\
=&\left(
\frac{\sin(l-t)/2}{\sin(l/2)}\right)^2\mu
+\frac{2\sin(t/2)\cdot\sin(l-t)/2}{\sin^2(l/2)}\sqrt{\frac{d\mu_1}{d\mu}}\,\mu
+\left(\frac{\sin(t/2)}{\sin(l/2)}\right)^2\mu_1.\label{expression-g4}
\end{align}
The second term in the last is represented as
\begin{equation*}
\frac{2 \sin (t/2) \cos (l/2) \sin(l-t)/2}{\sin^2 l/2} \varphi(\mu_1,\mu) = a_3(t) \varphi(\mu_1,\mu),
\end{equation*}
since from Definitions \ref{normalizedgeometricmean} and \ref{ellfunction} one has
\begin{equation*}
\sqrt{\frac{d\mu_1}{d\mu}}\ \mu = \cos \left(\frac{l}{2}\right)\, \varphi(\mu_1,\mu).
\end{equation*}
On the other hand the first and third terms are written as $a_1(t) \mu$ and $a_2(t) \mu_1$, respectively. Therefore we obtain the form \eqref{expression-g}.
Since $\gamma(0)=\mu$, $\gamma(\ell(\mu, \mu_1))=\mu_1$, easy computations show us 
\begin{equation*}
a_1(t)+a_2(t)+a_3(t)=\int_M a_1(t)\,d \mu + a_2(t) d\,\mu_1 + a_3(t)\,d \varphi(\mu_1,\mu) = \int_M d\ \gamma(t) = 1.
\end{equation*}

Moreover, it is obvious that $a_i(t)\ge 0,\ i=1,2,3$ for $0\le t\le l<\pi$.
Hence, we conclude that $\gamma(t)$ belongs to $\mathcal{P}(M)$ for any $t\in [0,l]$, which means that $\gamma$ is the geodesic being inside $\mathcal{P}(M)$ and joining $\mu, \mu_1\in \mathcal{P}(M)$.
Thus, we obtain Theorem \ref{maincor} (i) and (ii).

\begin{rem}\label{tangentline_gm}
Equation \eqref{Fgeodesic_iv-2} implies that the tangent line of $\gamma(t)$ at $\gamma(0)=\mu$, which is a curve in $\mathcal{P}(M)$, passes through the normalized geometric mean $\varphi(\mu, \mu_1)$.
Now, we consider the geodesic $\gamma_-(t)=\gamma(l-t)$ which has inverse direction of $\gamma$.
Then, $\gamma_-(l)=\mu$ and $\displaystyle{\dot{\gamma}_-(0)=\frac{1}{\tan\frac{l}{2}}\left(\varphi(\mu_1, \mu)-\mu_1\right) }$.
Hence, similarly as $\gamma$, the tangent line of $\gamma_-(t)$ at $\gamma_-(0)=\gamma(l)=\mu_1$ also passes through $\varphi(\mu, \mu_1)$.
Thus, we obtain Theorem \ref{maincor} (iii) and more generally the following.
\end{rem}

\begin{thm}\label{geometricm_gm}
Let $\mu$, $\mu_1$ be points on a geodesic $\gamma$.
Let $L_{\mu}$ and $L_{\mu_1}$ be the tangent lines, tangent to $\gamma$ at $\mu$ and $\mu_1$, respectively.
Then, the lines $L_{\mu}$ and $L_{\mu_1}$ intersect and their intersection point is the normalized geometric mean $\varphi(\mu,\mu_1)$ of $\mu$ and $\mu_1$.
 (See Figure \ref{geochara_gm-fig}).
\end{thm}

\begin{figure}[htbp]
\includegraphics[width=8.5cm]{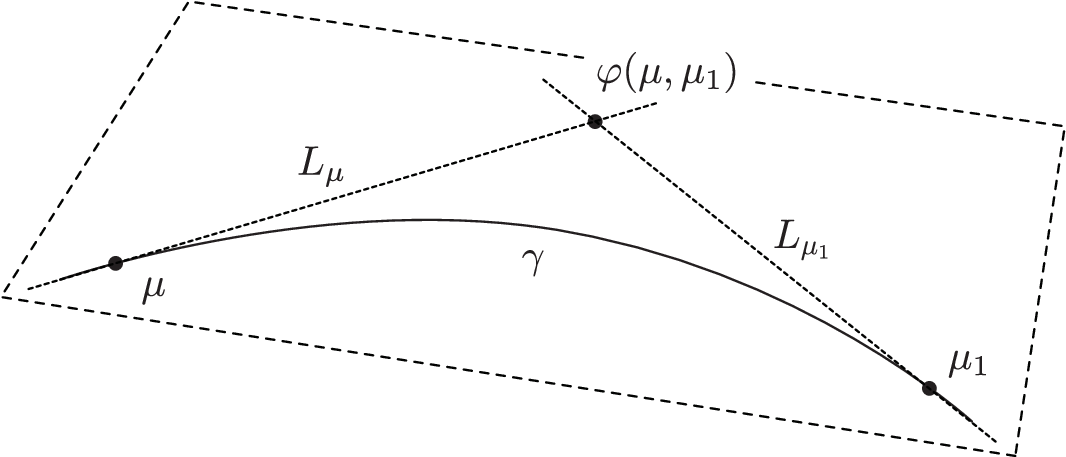}
\caption{a geometric characterization of $\varphi(\mu, \mu_1)$.}
\label{geochara_gm-fig}
\end{figure}

Substituting $t=l/2$ into \eqref{expression-g4}, we have
$$
\gamma(l/2)
=\left(
\frac{\sin(l/4)}{\sin(l/2)}\right)^2\left\{\mu
+2\sqrt{\frac{d\mu_1}{d\mu}}\,\mu+\mu_1\right\}
=\left(
\frac{1}{2\cos(l/4)}\right)^2\left\{1+\sqrt{\frac{d\mu_1}{d\mu}}\right\}^2\mu,
$$
from which we obtain Theorem \ref{maincor} (iv).

\begin{rem}
All the above arguments concerning with geodesics, the map $\varphi$ and the function $\ell$ are completely valid for the space $\mathcal{P}^{\infty}(M)$ of probability measures with smooth density function.
$\mathcal{P}^{\infty}(M)$  is dense in the space  $\mathcal{P}(M)$ (see Lemma \ref{dense}). 
\end{rem}

\section{Continuity of the map $\varphi$ and the function $\ell$}\label{continuity}

In this section we will show the following result.

\begin{prop}
Relative to the 
$\|\sqrt{\cdot}-\sqrt{\cdot}\|_{L_2}$--topology,
\begin{enumerate}
\item $\varphi : \mathcal{P}(M)\times \mathcal{P}(M) \rightarrow \mathcal{P}(M)$ is continuous and
\item $\ell : \mathcal{P}(M)\times \mathcal{P}(M) \rightarrow [0,\pi)$ is continuous.
\end{enumerate}
\end{prop}

\begin{proof}
We will show first (ii). Since the function $\rm{arccosine}$ is continuous, it suffices to verify that 
\begin{equation}
\cos \frac{\ell(\mu,\mu_1)}{2} = \int_M \sqrt{p(x) p_1(x)} d\lambda\qquad
(\mu=p(x)\lambda,\ \mu_1=p_1(x)\lambda)
\end{equation}
is continuous.
For this we find the following with another pair of measures $\mu'=p'(x)\lambda$, $\mu_1'=p_1'(x)\lambda$ of $\mathcal{P}(M)$, by applying the Cauchy-Schwarz inequality
\begin{align}\label{cosellfunction}
\left\vert \cos \frac{\ell(\mu,\mu_1)}{2}-\cos \frac{\ell(\mu',\mu_1')}{2}\right\vert
\leq& 
\int_M \left(\sqrt{p}\left\vert\sqrt{p_1}-\sqrt{p_1'}\right\vert + \sqrt{p_1'}\left\vert\sqrt{p}-\sqrt{p'}\right\vert \right)d\lambda \\ \nonumber
\leq& \left\|\sqrt{p}-\sqrt{p'}\right\|_{L_2} + \left\|\sqrt{p_1}-\sqrt{p'_1}\right\|_{L_2}.
\end{align}
From this it follows that $\cos \ell(\mu,\mu_1)/2$ is continuous.

We will next see that the map $\varphi$ is continuous.
As same as just above, let $\mu=p(x)\lambda$, $\mu_1=p_1(x)\lambda$, $\mu'=p'(x)\lambda$ and $\mu_1'=p_1'(x)\lambda$\, $\in \mathcal{P}(M)$.
We write $\varphi(\mu,\mu_1) = P(x)\lambda$ and $\varphi(\mu',\mu'_1) = P'(x)\lambda$, where 
\begin{equation}
P(x)= \frac{\sqrt{p(x)p_1(x)}}{\int_M \sqrt{p(x)p_1(x)}d\lambda},\quad
P'(x)= \frac{\sqrt{p'(x)p'_1(x)}}{\int_M \sqrt{p'(x)p'_1(x)}d\lambda}.
\end{equation}
We have then, by using the inequality $\left\vert \sqrt{a}-\sqrt{b}\right\vert^2 \leq \vert a - b\vert$ for any $a,b\geq 0$
\begin{equation}
\left\| \sqrt{P}- \sqrt{P'}\right\|_{L^2}^2 = \int_M\left(\sqrt{P}- \sqrt{P'}\right)^2 d\lambda \leq \int_M \left\vert P(x)- P'(x)\right\vert d \lambda.
\end{equation}
Here
\begin{align}
P(x)-P'(x)
=& \frac{\sqrt{p(x)p_1(x)}- \sqrt{p'(x)p_1'(x)}}{\int_M \sqrt{pp_1}d\lambda}+\frac{\int_M\left(\sqrt{p'p'_1}-\sqrt{pp_1}\right)d\lambda}{\int_M\sqrt{pp_1}d\lambda \int_M\sqrt{p'p'_1}d\lambda}\, \sqrt{p'(x)p_1'(x)},
\end{align}
so
\begin{multline}
\left\vert P(x)-P'(x)\right\vert
\leq\frac{\sqrt{p(x)}\left\vert\sqrt{p_1(x)}-\sqrt{p_1'(x)}\right\vert +\left\vert\sqrt{p(x)}-\sqrt{p'(x)}\right\vert\sqrt{p_1'(x)}}{\int_M\sqrt{pp_1}d\lambda} \\
+ \frac{\int_M\left\{\sqrt{p}\left\vert\sqrt{p_1}-\sqrt{p_1'}\right\vert+\left\vert\sqrt{p}-\sqrt{p'}\right\vert\sqrt{p_1'}\right\}d\lambda}{\int_M\sqrt{pp_1}d\lambda\, \int_M \sqrt{p'p'_1} d\lambda}\, \sqrt{p'(x)p_1'(x)}
\end{multline}
and hence
\begin{multline*}
\int_M\left\vert P(x)-P'(x)\right\vert d \lambda \\
\leq \frac{2}{\int_M \sqrt{pp_1}d\lambda}\,\int_M\left\{\sqrt{p(x)}\left\vert\sqrt{p_1(x)}-\sqrt{p_1'(x)}\right\vert + \left\vert\sqrt{p(x)}-\sqrt{p'(x)}\right\vert\sqrt{p_1'(x)}\right\} d\lambda.
\end{multline*}
From the Cauchy-Schwarz inequality one gets
\begin{equation}
\| \sqrt{P}-\sqrt{P'}\|_{L^2}^2 \leq \frac{2}{\int_M \sqrt{pp_1} d\lambda}\left(\left\|\sqrt{p}-\sqrt{p'}\right\|_{L_2}+ \left\|\sqrt{p_1}-\sqrt{p_1'}\right\|_{L_2}\right)
\end{equation}
which indicates that $\varphi$ is continuous.
\end{proof}

\section{Riemannian distance function of $(\mathcal{P}(M), G)$}
\label{exponential}

In this section we will exhibit that $\ell(\mu, \mu_1)$ is precisely the Riemannian distance of $\mu$ and $\mu_1$ in $\mathcal{P}(M)$.
For this purpose we first restrict our argument to $\mathcal{P}^{\infty}(M)$, 
the space of probability measures with smooth density function.
We define the exponential map over $\mathcal{P}^{\infty}(M)$.
We prove, then, Gauss lemma, the existence of a totally normal neighborhood in $\mathcal{P}^{\infty}(M)$ with respect to the Fisher metric $G$ and show that $\ell(\mu, \mu_1)$ gives the Riemannian distance in $\mathcal{P}^{\infty}(M)$ for $\mu, \mu_1\in \mathcal{P}^{\infty}(M)$.
We prove secondly that $\mathcal{P}^{\infty}(M)$ is dense in $\mathcal{P}(M)$ with respect to the $C^0$-norm (Lemma \ref{dense}) so that the Riemannian distance of $\mu,\mu_1\in \mathcal{P}^{\infty}(M)$ in the space $\mathcal{P}(M)$ is actually given by the function $\ell(\mu,\mu_1)$ by the aid of reductio ad absurdum.
Finally we verify that  $\ell(\mu,\mu_1)$ is properly the Riemannian distance of $\mu,\mu_1$ in the space $\mathcal{P}(M)$.

For the sake of convenience we provide $\mathcal{P}^{\infty}(M)$ an $H^a_1$-topology, $a > n$, $n = \dim M$. We equip the compact manifold $M$ with a Riemannian metric  whose Riemannian volume form coincides with the measure $\lambda$. 
The Sobolev norm $\|\cdot\|_{H^a_1}$ 
is defined by $\| f\|_{H^a_1}:= \| f\|_{L_a}+ \| \nabla f \|_{L_a}$, $f\in C^{\infty}(M)$.
From the Sobolev embedding theorem there exists a constant $C(a)>0$ such that for all $f\in C^{\infty}(M)$ $\displaystyle \|f\|_{C^0}(:=\sup_{x\in M} \vert f(x)\vert) \leq C(a) \|f\|_{H^a_1}$.
See \cite[\S 7 and 2.22\, (11)]{Aubin}.
Notice that the $\|\sqrt{\cdot}-\sqrt{\cdot}\|_{(L_2,\lambda)}$-norm is related to the $H^1_a$-norm from  H\"older inequality as
 \begin{eqnarray*}
 \|\sqrt{p}-\sqrt{p_1}\|_{(L_2,\lambda)} \leq \|p-p_1\|_{L_1}^{1/2}\leq \|p-p_1\|_{L_a}^{1/2}\leq \|p-p_1\|_{H^1_a}^{1/2}.
 \end{eqnarray*}
 

\subsection{Exponential map on $\mathcal{P}^{\infty}(M)$}

Let $\mu\in\mathcal{P}^{\infty}(M)$.
Let $\tau\in T_\mu\mathcal{P}^{\infty}(M)$ be a tangent vector at $\mu\in\mathcal{P}^{\infty}(M)$ and suppose that there exists a geodesic $\gamma : [0,1] \rightarrow \mathcal{P}^{\infty}(M)$ satisfying $\gamma(0) = \mu, \dot{\gamma}(0) = \tau$.
Then $\gamma(1)\in\mathcal{P}^{\infty}(M)$ will be customarily denoted by $\exp_\mu \tau$.
The geodesic $\gamma$ can thus be written by
\begin{equation*}
\gamma(t) = \exp_\mu t\tau.
\end{equation*}

\begin{lem}\label{Fgeodesic_4}
For any $\mu_1\in\mathcal{P}^{\infty}(M)$, $\mu_1 \ne \mu$,
there exists a geodesic $\gamma : [0,1] \rightarrow \mathcal{P}^{\infty}(M)$  satisfying $\gamma(0) = \mu, \gamma(1) = \mu_1$ by setting
\begin{equation}\label{geodesic_4}
\gamma(t)=\exp_\mu t\tau
=\left(\cos l\hspace{0.5mm}\frac{t}{2}+\frac{1}{l}\cdot\sin l\hspace{0.5mm}\frac{t}{2}\cdot\frac{d\tau}{d\mu}\right)^2\mu,
\end{equation}
where $\tau\in T_\mu\mathcal{P}^{\infty}(M)$ defined by $\tau = l\tilde{\tau}, l = \ell(\mu, \mu_1)$ and $\tilde{\tau}\in T_\mu\mathcal{P}^{\infty}(M)$ is a unit tangent vector defined by
\begin{equation}\label{geodesic_4_iv}
\tilde{\tau} =\frac{1}{\tan\frac{l}{2}}\left(\varphi(\mu, \mu_1)-\mu
\right).
\end{equation}
\end{lem}

\begin{proof}
From Proposition \ref{densityfree},
\begin{equation*}
\tilde{\gamma}(t) =\left(\cos\frac{t}{2} + \sin\frac{t}{2}\frac{d\tilde{\tau}}{d\mu}\right)^2\mu
\end{equation*}
with $\tilde{\tau}$ of \eqref{geodesic_4_iv} gives us a geodesic, parametrized by arc-length, satisfying $\tilde{\gamma}(0) = \mu$, $\dot{\tilde{\gamma}}(0) = \tilde{\tau}$ and $\tilde{\gamma}(l) = \mu_1$.

Put $\tau = l\tilde{\tau}$ and $t = ls$ and set $\gamma(s) = \tilde{\gamma}(ls)$.
Then, $\gamma(s)$ is a geodesic defined over $[0, 1]$, which has the form \eqref{geodesic_4}.
It is straightforward to see that $\gamma(0) = \mu, \gamma(1) = \mu_1$ and $\dot{\gamma}(0) = \tau$.
\end{proof}

Let $\mu \in \mathcal{P}^{\infty}(M)$ be a probability measure of positive smooth density function.
We fix $\mu$ for a moment.
Let $\varepsilon$ be a real number satisfying $0 < \varepsilon < \pi$ and let $B(\mu;\varepsilon)$ be a set of probability measures $\mu_1\in \mathcal{P}^{\infty}(M)$ satisfying $\ell(\mu,\mu_1) < \varepsilon$:
\begin{equation}\label{neighborhood}
B(\mu;\varepsilon) := \{\mu_1\in \mathcal{P}^{\infty}(M)\,|\, \ell(\mu,\mu_1) < \varepsilon\}.
\end{equation}

Let $0<\varepsilon_1<\pi$ and set
\begin{equation*}
\mathscr{B}(\mu; \varepsilon_1) :=
\left\{\tau \in T_\mu\mathcal{P}^{\infty}(M)\,\left|\,
|\tau|_{\mu} < \varepsilon_1,\ \inf_{x\in M}\frac{d\tau}{d\mu}(x)>-|\tau|_{\mu}\cot\frac{|\tau|_{\mu}}{2}
\right.
\right\}.
\end{equation*}

Note that when $\vert\tau\vert_{\mu}= 0$ we put $\displaystyle{\vert\tau\vert_{\mu}\, \cot \frac{\vert\tau\vert_{\mu}}{2} = 2}$. 
Take a Riemannian metric $g$ on $M$ whose Riemannian volume form $dv_g$ coincides with the measure $\mu$.
Then $\displaystyle\vert \tau\vert_{\mu} \leq \left\|\frac{d\tau}{d\mu}\right\|_{H^a_1}$ so that the map $\tau \mapsto \vert\tau\vert_{\mu}$ is continuous with respect to the $H^a_1$--topology.
Moreover, the inequality $\displaystyle\inf_{x\in M}\frac{d\tau}{d\mu}(x)>-|\tau|_{\mu}\cot\dfrac{|\tau|_{\mu}}{2}$ is also an open Sobolev norm condition in the following way.
Set $f = \dfrac{d\tau}{d\mu}$ and $f_- := \dfrac{f-\vert f\vert}{2}$.
Then $f_-(x) \leq 0$ for all $x\in M$ and $f_- \in C^0(M)$ so that the inequality is equivalent to the $C^0$-norm inequality: $\displaystyle{\|f_-\|_{C^0} < \vert\tau\vert_{\mu}\cot\frac{|\tau|_{\mu}}{2}
}$.
By using the mollifiers whose definition will be given at the proof of Lemma \ref{dense} one has a family of smooth functions $f_{-,s} \in C^{\infty}(M)$, $s>0$ for $f_-$ such that $\|f_{-,s}-f_-\|_{C^0} \rightarrow 0$ as $s \rightarrow 0$.
Therefore, the inequality with respect to the $H^a_1$-norm
\begin{equation*}
\|f_{-,s}\|_{H^a_1} < \frac{1}{C(a)} \vert\tau\vert_{\mu}\, \cot \frac{\vert\tau\vert_{\mu}}{2}
\end{equation*}
implies the required inequality, by the aid of the Sobolev embedding theorem,
since
\begin{equation}
\|f_-\|_{C^0} \leq \|f_{-,s}-f_-\|_{C^0}+ \|f_{-,s}\|_{C^0} <
\|f_{-,s}-f_-\|_{C^0}+ \vert\tau\vert_{\mu}\, \cot \frac{\vert\tau\vert_{\mu}}{2}
\end{equation}
in which the term $\|f_{-,s}-f_-\|_{C^0}$ is taken small as possible.

\begin{prop}\label{BtoB}
The exponential map $\exp_\mu : \mathscr{B}(\mu; \varepsilon) \rightarrow B(\mu;\varepsilon)$ defined by
\begin{equation*}
\exp_\mu\,\tau=\left(\cos\frac{|\tau|_{\mu}}{2}+\frac{1}{|\tau|_{\mu}}\sin\frac{|\tau|_{\mu}}{2}\cdot\frac{d\tau}{d\mu}\right)^2\mu
\end{equation*}
is a bijection.
\end{prop}

\begin{proof}
First we will show that $\exp_\mu \tau$ which we denoted by $\mu_1$ belongs to $B(\mu; \varepsilon)$ for any $\tau \in\mathscr{B}(\mu; \varepsilon)$.
Since 
\begin{equation*}
\sqrt{\frac{d\mu_1}{d\mu}} =\left(
\cos\frac{|\tau|_\mu}{2}+\frac{1}{|\tau|_\mu}\sin\frac{|\tau|_\mu}{2}\cdot\frac{d\tau}{d\mu}\right),
\end{equation*}
we have
\begin{equation*}
\begin{split}
\int_M\sqrt{\frac{d\mu_1}{d\mu}}\,d\mu
=&\int_M\left(
\cos\frac{|\tau|_\mu}{2}+\frac{1}{|\tau|_\mu}\sin\frac{|\tau|_\mu}{2}\cdot\frac{d\tau}{d\mu}\right)d\mu\\
=&\cos\frac{|\tau|_\mu}{2}\int_Md\mu
+\frac{1}{|\tau|_\mu}\sin\frac{|\tau|_\mu}{2}\int_Md\tau
=\cos\frac{|\tau|_\mu}{2}.
\end{split}
\end{equation*}
Then, $\displaystyle\cos\frac{|\tau|_\mu}{2} =\cos\frac{\ell(\mu_1,\mu)}{2}$ from \eqref{distancefunction} and hence $|\tau|_\mu = \ell(\mu_1,\mu)$ and thus $\mu_1 \in B(\mu; \varepsilon)$.

Next we will show that the map $\exp_\mu$ is injective over $\mathscr{B}(\mu; \varepsilon)\backslash\{0\}$.
Let $\tau, \tau' \in \mathscr{B}(\mu; \varepsilon)\backslash\{0\}$.
Assume that $\exp_\mu \tau =\exp_\mu \tau'$ which we denote by $\mu_1$.
Then from the above argument, we have $\ell(\mu_1, \mu) = |\tau|_\mu = |\tau'|_\mu$.
Moreover, from
\begin{align*}
\mu_1
=&\left(
\cos\frac{|\tau|_\mu}{2}+\frac{1}{|\tau|_\mu}\sin\frac{|\tau|_\mu}{2}\cdot\frac{d\tau}{d\mu}\right)^2\mu
=\left(
\cos\frac{|\tau'|_\mu}{2}+\frac{1}{|\tau'|_\mu}\sin\frac{|\tau'|_\mu}{2}\cdot\frac{d\tau'}{d\mu}\right)^2\mu,
\end{align*}
it follows similarly as in the proof of Theorem \ref{uniquenessthm} that $d\tau/d\mu= d\tau'/d\mu$ on $M$ and hence $\tau = \tau'$, which means the injectivity of the map $\exp_\mu$.

The surjectivity is obtained by taking $\mu_1$ in $B(\mu,\varepsilon)$ and also $\tau = l\tilde{\tau}\in T_\mu\mathcal{P}^{\infty}(M)$, where $\tilde{\tau}=\frac{1}{\tan(l/2)}\left(\varphi(\mu, \mu_1)-\mu\right)$ is a unit tangent vector at $\mu$ and $l=\ell(\mu,\mu_1)$.
Then, from Lemma \ref{Fgeodesic_4} $\mu_1$ is described as $\mu_1 =\exp_\mu \tau$, which implies the surjectivity of $\exp_{\mu}$.
\end{proof}

\begin{rem}\label{smoothness}
From the above proposition, especially from its actual form the map $\exp_\mu$ is smooth over $\mathscr{B}(\mu; \varepsilon)\backslash\{0\}$ together with smooth inverse map $\exp^{-1}_\mu$. For the smoothness refer to \cite[II]{Lang}.
\end{rem}

\subsection{A totally normal neighborhood}

\begin{lem}\label{lem_L2compere}
Let $\mu = p(x)\,\lambda$ and $\mu_1 = p_1(x)\,\lambda$ be probability measures in $\mathcal{P}^{\infty}(M)$.
Then,
\begin{equation}\label{L2compere}
\ell(\mu,\mu_1) < \varepsilon\ \Longleftrightarrow \left\|\sqrt{p_1} - \sqrt{p}\right\|_{L_2} < \sqrt{2}\sqrt{1 -\cos\frac{\varepsilon}{2}},
\end{equation}
and hence, $B(\mu; \varepsilon)$ is written as
\begin{equation}\label{nbd}
B(\mu;\varepsilon)=\left\{\mu_1 = p_1(x)\,\lambda\,
\left|\,\left\|\sqrt{p_1} - \sqrt{p}\right\|_{L_2}<\sqrt{2}\sqrt{1-\cos\frac{\varepsilon}{2}}\,\right.\right\}.
\end{equation}
\end{lem}

\begin{rem}
From \eqref{nbd} $B(\mu;\varepsilon)$ can be regarded as a neighborhood of $\mathcal{P}^{\infty}(M)$ with respect to the $\|\sqrt{\cdot}-\sqrt{\cdot}\|_{L^2}$-norm around $\mu = p \lambda$.
Therefore we consider each $B(\mu;\varepsilon)$ as a neighborhood of $\mathcal{P}^{\infty}(M)$ around $\mu$.
\end{rem}

\begin{proof}
Denote $\ell(\mu,\mu_1)$ by $\ell$ by abbreviation.
Then, the left hand side of \eqref{L2compere} is equivalent to $0 \le \ell/2 < \varepsilon/2$ and hence to $\cos(\varepsilon/2) <\cos(\ell/2) \le 1$.
On the other hand, we have the following identity:
\begin{equation}
\left\|\sqrt{p_1} - \sqrt{p}\right\|_{L_2}^2=2-2\cos\frac{\ell}{2}
\end{equation}
which is derived from
\begin{equation}
\left\|\sqrt{p_1} - \sqrt{p}\right\|_{L_2}^2
 = \int_M\left(\sqrt{p_1} - \sqrt{p}\right)^2\,d\lambda
 =2-2\int_M\sqrt{p_1 p}\,d\lambda,
\end{equation}
where $\displaystyle \int_M\sqrt{p_1 p}\,d\lambda$ is represented by
\begin{equation*}
\int_M\sqrt{\frac{p_1}{p}}\,p\,d\lambda=\int_M\sqrt{\frac{d\mu_1}{d\mu}}\,d\mu
=\cos\frac{\ell(\mu, \mu_1)}{2}.
\end{equation*}

Then, $\cos(\varepsilon/2)<\cos(l/2)\le 1$ is equivalent to
\begin{align*}
\cos\frac{\varepsilon}{2} < 1-\frac{1}{2}\left\|\sqrt{p_1} - \sqrt{p}\right\|_{L_2}^2
 \le 1
\Longleftrightarrow\ &
1-\cos\frac{\varepsilon}{2}>\frac{1}{2}\left\|\sqrt{p_1} - \sqrt{p}\right\|_{L_2}^2
\ge 0\\
\Longleftrightarrow\ &
2\left(1-\cos\frac{\varepsilon}{2}\right)>\left\|\sqrt{p_1} - \sqrt{p}\right\|_{L_2}^2
\ge 0
\end{align*}
from which it holds \eqref{L2compere}.
Notice that $0 \le \varepsilon/2 < \pi/2$.
\end{proof}

Let $\mu_1 = p_1(x)\,\lambda, \mu_2 = p_2(x)\,\lambda \in B(\mu; \varepsilon)$ be arbitrary probability measures.
From Lemma \ref{lem_L2compere} we have
\begin{equation*}
\left\|\sqrt{p_i} - \sqrt{p}\right\|_{L_2} < \sqrt{2}\sqrt{1-\cos\frac{\varepsilon}{2}},\qquad i=1,2.
\end{equation*}
From the triangle inequality with respect to the $L_2$-norm, we have then
\begin{equation}\label{ineq_tri1}
\left\| \sqrt{p_2} - \sqrt{p_1}
\right\|_{L_2} \le\left\|\sqrt{p_2} - \sqrt{p}\right\|_{L_2} +\left\|\sqrt{p_1} - \sqrt{p}\right\|_{L_2}<2\sqrt{2}\sqrt{1-\cos\frac{\varepsilon}{2}}.
\end{equation}

\begin{lem}
Let $t$ be a real number satisfying $0 < t < \pi/2$.
Then, we have
\begin{equation}\label{ineq_tri2}
\sqrt{2}\sqrt{1-\cos t}\le \sqrt{1-\cos 2t}.
\end{equation}
\end{lem}

\begin{proof}
From the obvious equality $1-\cos 2t = 2(1-\cos^2 t)$, we have
\begin{equation*}
\sqrt{1-\cos 2t} = \sqrt{2}\sqrt{1-\cos^2 t}.
\end{equation*}
Since $1 -\cos t > 0$ and $1 +\cos t > 1$ for $0 < t < \pi/2$, we have
\begin{equation*}
\sqrt{2}\sqrt{1-\cos t} < \sqrt{2} \sqrt{1-\cos t}\sqrt{1+\cos t}=\sqrt{2}\sqrt{1-\cos^2 t}
\end{equation*}
which is equal to $\sqrt{1 -\cos 2t}$.
\end{proof}

Now, let $B(\mu; \varepsilon)$ be a neighborhood around $\mu$, defined at \eqref{neighborhood} with $\varepsilon < \pi/4$ and take arbitrary probability measures $\mu_i= p_i\,\lambda\in B(\mu; \varepsilon), i=1,2$.
Then, from \eqref{ineq_tri1} and \eqref{ineq_tri2}, we have
\begin{align*}
\left\|\sqrt{p_2} - \sqrt{p_1}\right\|_{L_2}
< & 2\sqrt{2}\sqrt{1-\cos\frac{\varepsilon}{2}}
\le 2\sqrt{1-\cos\varepsilon}\\
\le & \sqrt{2}\sqrt{1-\cos 2\varepsilon}= \sqrt{2}\sqrt{1-\cos\frac{4\varepsilon}{2}}.
\end{align*}

\begin{lem}\label{lem-ineq_1}
Let $B(\mu; \varepsilon)$ be a neighborhood with $\varepsilon < \pi/4$.
Then, for any $\mu_1, \mu_2 \in B(\mu; \varepsilon)$,
\begin{equation}\label{ineq_l}
\ell(\mu_1, \mu_2) < 4\varepsilon.
\end{equation}
\end{lem}

\begin{proof}
For $\mu_1, \mu_2 \in B(\mu; \varepsilon)$ one has $\ell(\mu_i, \mu) < \varepsilon$, $i = 1, 2$, equivalently
\begin{equation*}
\left\| \sqrt{p_i} - \sqrt{p}\right\|_{L_2} < \sqrt{2}\sqrt{1-\cos\frac{\varepsilon}{2}},\qquad i=1,2
\end{equation*}
and from the above argument
\begin{equation*}
\left\|\sqrt{p_2} - \sqrt{p_1}\right\|_{L_2} < \sqrt{2}\sqrt{1-\cos\frac{4\varepsilon}{2}}
\end{equation*}
which means \eqref{ineq_l}.
\end{proof}

\begin{prop}
Let $\mu \in\mathcal{P}^{\infty}(M)$ be an arbitrary probability measure and $\varepsilon$ be a real number satisfying $0 < \varepsilon < \pi$.
Let $W = B(\mu; \varepsilon/4)$ be a neighborhood defined at \eqref{neighborhood}.
For any $\mu_1 \in W$, let $B(\mu_1; \varepsilon)$ be a neighborhood around $\mu_1$.
Then,
\begin{enumerate}
\item $W \subset B(\mu_1; \varepsilon)$ and
\item $\exp_{\mu_1}$ is a diffeomorphism between $\mathscr{B}(\mu_1; \varepsilon)$ and $B(\mu_1; \varepsilon)$.
\end{enumerate}
\end{prop}

The neighborhood $W$ is called a \textit{totally normal neighborhood} of $\mu$.

\begin{proof}Notice $\ell(\mu,\mu_1) < \varepsilon/4$. 
If $\mu_2 \in W$, then $\ell(\mu, \mu_2) < \varepsilon/4$.
From Lemma \ref{lem-ineq_1} we have $\ell(\mu_1,\mu_2) < 4\cdot(\varepsilon/4) = \varepsilon$ and hence $\mu_2 \in B(\mu_1; \varepsilon)$.
Since $\mu_2 \in W$ is arbitrary, we see $W\subset B(\mu_1; \varepsilon)$.

Assertion (ii) is shown from Proposition \ref{BtoB} together with Remark \ref{smoothness}, since $0 < \varepsilon < \pi$.
\end{proof}

\begin{lem}[Gauss Lemma]\label{gausslemma}
Denote by $f(t,\tau)$ the image of the exponential map $\exp_\mu t \tau$,
$t > 0$ and $\tau \in T_\mu\mathcal{P}(M)$ of unit norm $\vert \tau\vert_{\mu} =$ $G_\mu(\tau,\tau)^{1/2} = 1$.
Then
\begin{equation*}
G_{f(t,\tau)}\left(\frac{\partial f}{\partial t}, \frac{\partial f}{\partial \tau}_{\ast}(\delta\tau )\right)=0,
\end{equation*}
where $\frac{\partial f}{\partial t}$ is the differential of $f$ with respect to $t$ and $\frac{\partial f}{\partial \tau}_{\ast}$ is the differential map from $T_{\tau}S_{\mu}$ to $T_{f(t,\tau)}{\mathcal P}(M)$. 
Here $S_\mu :=\{\sigma\in T_\mu\mathcal{P}^{\infty}(M)\,|\,G_\mu(\sigma, \sigma) = 1\}$ and $\delta\tau$ is a tangent vector at $\tau$ to $S_{\mu}$.
\end{lem}

\begin{proof}
While this lemma is a routine in Riemannian geometry,
we verify it directly.
Since $\displaystyle f(t, \tau)=\left(\cos\frac{t}{2}+\sin\frac{t}{2}\frac{d\tau}{d\mu}\right)^2\,\mu$,
we have
\begin{align*}
\frac{\partial f}{\partial t}=& \left(\cos\frac{t}{2}+\sin\frac{t}{2}\frac{d\tau}{d\mu}\right)\left(-\sin\frac{t}{2}+\cos\frac{t}{2}\frac{d\tau}{d\mu}\right)\,\mu,\\
\frac{\partial f}{\partial \tau}_{\ast}(\delta \tau)=&2\left(\cos\frac{t}{2}+\sin\frac{t}{2}\frac{d\tau}{d\mu}\right)\sin\frac{t}{2}\cdot\frac{d(\delta\tau)}{d\mu}\,\mu.
\end{align*}

Now we will see $G_{f(t,\tau)}(\partial f/\partial t, \partial f/\partial \tau_{\ast}(\delta \tau))=0$.
Since
\begin{align*}
\frac{d(\partial f/\partial t)}{d\, f(t,\tau)}
=& \frac{\left(\cos\frac{t}{2} + \frac{d\tau}{d\mu}\,\sin\frac{t}{2}\right)\left(-\sin\frac{t}{2} + \frac{d\tau}{d\mu}\,\cos\frac{t}{2}\right)}{\left(\cos\frac{t}{2} + \frac{d\tau}{d\mu}\,\sin\frac{t}{2}\right)^2}
= \frac{-\sin\frac{t}{2} + \frac{d\tau}{d\mu}\,\cos\frac{t}{2}}{\cos\frac{t}{2} + \frac{d\tau}{d\mu}\,\sin\frac{t}{2}}
\end{align*}
and similarly
\begin{equation*}
\frac{d(\partial f/\partial \tau_{\ast}(\delta\tau))}{d\, f(t,\tau)}
= \frac{2\frac{d(\delta \tau)}{d\mu} \sin\frac{t}{2}}{\cos\frac{t}{2} + \frac{d\tau}{d\mu}\,\sin\frac{t}{2}}
\end{equation*}
and thus
\begin{equation*}
G_{f(t,\tau)}\left(\frac{\partial f}{\partial t}, \frac{\partial f}{\partial \tau}_{\ast}(\delta\tau)\right)
= \int_M \frac{2\left(-\sin\frac{t}{2} + \frac{d\tau}{d\mu}\,\cos\frac{t}{2}\right)\,\frac{d(\delta \tau)}{d\mu} \sin\frac{t}{2}}
{\left(\cos\frac{t}{2} + \frac{d\tau}{d\mu}\,\sin\frac{t}{2}\right)^2}
\cdot \left(\cos\frac{t}{2} + \frac{d\tau}{d\mu}\,\sin\frac{t}{2}\right)^2 d\mu
\end{equation*}
which is reduced to zero, since
\begin{multline*}
\int_M 2\left(-\sin\frac{t}{2}+\cos\frac{t}{2}\cdot\frac{d\tau}{d\mu}\right)\sin\frac{t}{2}\cdot\frac{d(\delta\tau)}{d\mu}\,d\mu\\
=- 2\sin^2\frac{t}{2}\int_M\frac{d(\delta\tau)}{d\mu}d\mu
+2\sin\frac{t}{2}\,\cos\frac{t}{2}\,G_\mu(\tau, \delta\tau)=0,
\end{multline*}
where $G_{\mu}(\tau, \delta\tau)=0$ is derived from the derivation of $G_{\mu}(\tau, \tau)=1$ along the direction $\delta\tau$.
Thus, the lemma is proved.
\end{proof}

\begin{prop}\label{prop_length}
Let $\mu \in \mathcal{P}^{\infty}(M)$ and $\varepsilon \in (0,\pi)$.
Let $\mathscr{B}(\mu; \varepsilon)$ be an $\varepsilon$-open neighborhood in $T_\mu\mathcal{P}^{\infty}(M)$ such that $B(\mu; \varepsilon) =\exp_\mu (\mathscr{B}(\mu; \varepsilon))$.
Let $\gamma : [0, 1] \rightarrow B(\mu; \varepsilon)$ be a geodesic segment satisfying $\gamma(0) = \mu$.

If $c : [0,1] \rightarrow \mathcal{P}^{\infty}(M)$ be any piecewise $C^1$-curve joining $\gamma(0)$ and $\gamma(1)$, then the length of $\gamma$ and $c$ satisfies
\begin{equation*}
\mathscr{L}(\gamma) \le \mathscr{L}(c)
\end{equation*}
and if equality holds, then $\gamma([0, 1]) = c([0, 1])$,
that is, the image by $\gamma$ of $[0,1]$ coincides with the image by $c$ of $[0,1]$.
\end{prop}

\begin{proof}
We may suppose that $c([0,1])\subset B(\mu; \varepsilon)$.
Since $\exp_\mu$ is bijective on $B(\mu; \varepsilon)$,
$c(t)$ for $t (\ne 0)$ can be written uniquely as
\begin{equation*}
c(t) =\exp_\mu (r(t) \tau (t))
\end{equation*}
where $t\mapsto\tau(t)$ is a piecewise $C^1$-curve in $T_\mu\mathcal{P}^{\infty}(M)$ with $|\tau(t)|_{G,\mu} =1$ and $r:(0,1]\rightarrow\mathbb{R}$ is a positive piecewise $C^1$-function.

 By setting $f(r,\tau) = \exp_{\mu}(r\tau)$, we write $c(t)$ as $c(t) = f(r(t),\tau(t))$ for any $t(\ne 0)$.
It follows then that, except for a finite number of points
\begin{equation*}
\frac{dc}{dt}(t)=\frac{\partial f}{\partial r}\,\dot{r}(t)+\frac{\partial f}{\partial \tau}_{\ast}\left(\frac{d\tau}{dt}\right).
\end{equation*}
Here $\displaystyle \frac{d\tau}{dt}\in T_{\tau(t)}S_{\mu}$ is the velocity vector of the curve $\tau(t)$.
From Lemma \ref{gausslemma} two vectors of the right hand side are orthogonal each other with respect to the metric $G$ and $\displaystyle \left|\frac{\partial f}{\partial r}\right|_{c(t)}=1$ with respect to $G$. 
Then,
\begin{equation*}
\left|\frac{d c}{d t}\right|^2_{c(t)}
=|\dot{r}(t)|^2+\left|\frac{\partial f}{\partial \tau}_{\ast}\left(\frac{d\tau}{dt}\right)
\right|^2_{c(t)}\ge |\dot{r}(t)|^2.
\end{equation*}
Therefore, for a sufficiently small positive real number $\delta$, we have
\begin{equation*}
\int^1_\delta \left|\frac{d c}{dt}(t)\right|_{c(t)}\,dt
\ge\int^1_\delta |\dot{r}(t)|\,dt
\ge\left|\int^1_\delta \dot{r}(t)\,dt\right|
\ge  r(1)-r(\delta).
\end{equation*}
Taking $\delta \rightarrow 0$, we obtain $\mathscr{L}(c) \ge \mathscr{L}(\gamma)$,
because $r(1) = \ell(\gamma(1), \mu) = \mathscr{L}(\gamma)$.

If $c([0,1])$ is not contained in $B(\mu; \varepsilon)$, we consider the first point $t_1\in (0,1)$ for which $c(t_1)$ belongs to the boundary of $B(\mu; \varepsilon)$.
We have then
\begin{equation*}
\mathscr{L}(c) \geq \mathscr{L}(c_{\vert[0,t_1]}) \geq \varepsilon > \mathscr{L}(\gamma).
\end{equation*}
\end{proof}
Refer to \cite[Chap.3, sec. 3]{doC} and \cite[II, \S 10]{Milnor} for a proof for a {\it finite} dimensional Riemannian manifold.
 
\begin{thm}\label{smoothdistancetheorem}
Let $c : [a,b] \rightarrow \mathcal{P}^{\infty}(M)$ be a piecewise $C^1$-curve with a parameter proportional to arc length.
If $c$ has length less than or equal to the length of any other piecewise $C^1$-curve joining $c(a)$ to $c(b)$, then $c$ is a geodesic.
\end{thm}

\begin{proof}
Let $t \in [a,b]$ and let $W$ be a totally normal neighborhood of a point $c(t)$.
Then, there exists a closed interval $I\subset [a,b]$, with non-empty interior and $t\in I$ such that $c(I)\subset W$.
The restriction $c|_I :I\rightarrow W$ is a piecewise $C^1$-curve joining two points of $W$.
From Proposition \ref{prop_length} together with the hypothesis, the length of $c|_I$ is equal to the length of a radial geodesic joining these two points.
From Proposition \ref{prop_length} and from the fact that $c|_I$ is parametrized proportionally to arc length, $c|_I$ is a geodesic.
\end{proof}

From this theorem we can assert that the function $\ell=\ell(\mu,\mu_1)$ gives the Riemannian distance in $\mathcal{P}^{\infty}(M)$ of $\mu$, $\mu_1\in\mathcal{P}^{\infty}(M)$.
Now we will achieve the final aim of this section.

\begin{thm}\label{distancetheorem}
The function $\ell=\ell(\mu,\mu_1)$ is actually the Riemannian distance in $\mathcal{P}(M)$ of $\mu$ and $\mu_1$ of $\mathcal{P}(M)$. 
\end{thm}

To obtain this theorem we first show the following.

\begin{lem}\label{dense}
$\mathcal{P}^{\infty}(M)$ is dense in $\mathcal{P}(M)$ with respect to the $C^0$-norm.
More precisely, if $f$ is a continuous function on $M$, then there exists a family of smooth functions $f_{\delta}$, $\delta >0$ such that $\| f_{\delta} - f \|_{C^0} \rightarrow 0$, as $\delta \rightarrow 0$. 
\end{lem}

\begin{proof}
Let $\{ \rho_{\alpha}\,|\,\alpha\in A\}$ be a partition of unity subordinate to an open covering $\{U_{\alpha}\,|\,\alpha\in A\}$ of a compact manifold $M$, $\dim M=n \geq 2$.
Here $A$ is a finite set.
We may assume that each $U_{\alpha}$ is a coordinate neighborhood diffeomorphic to a euclidean open ball in $\mathbb{R}^n$ and $\mathrm{supp}\,\rho_{\alpha}\subset V_{\alpha}$, ${\overline{V}}_{\alpha}$ is compact in $U_{\alpha}$. 

Let $f$ be a continuous function on $M$.
Set for each $\alpha$ $f_{\alpha}:= \rho_{\alpha}\, f$.
Then, ${\rm supp}\, f_{\alpha}\subset V_{\alpha}$.
We may extend the function $f_{\alpha}$ outside of  $V_{\alpha}$, as $f_{\alpha}(x)=0$, $x\in \mathbb{R}^n \setminus V_{\alpha}$.
Let $\{\psi_{\delta}\,|\,\delta>0\}$ be a family of functions which satisfies
\begin{enumerate}
\item $\psi_{\delta}(x) \geq 0$ for any $x\in \mathbb{R}^n$,
\item $\psi_{\delta}\in C^{\infty}(\mathbb{R}^n)$,
\item ${\rm supp}\ \psi_{\delta} = B_{\delta}(0)$, where $B_{\delta}(0)\subset \mathbb{R}^n$ is the euclidean closed ball of radius $\delta$ with center $0$ and
\item $\displaystyle \int_{\mathbb{R}^n} \psi_{\delta} dv = 1$.
\end{enumerate}
We call $\{\psi_{\delta}\}$ a sequence of  mollifiers.
We define such a sequence $\{\psi_{\delta}\}$  for instance by $\psi_{\delta}(x)=\delta^{-n} \psi(x/\delta)$, $\delta>0$, where $\psi(y)$ is a bump function given by $\displaystyle\psi(y)= c_n \exp\left\{1/(\|y\|^2-1)\right\}$ for $y\in\mathbb{R}^n$ of $\|y\| < 1 $ and $\psi(y) = 0$ for $y$ of $\|y\| \geq 1$.
Here $c_n$ is a normalization constant according to (iv).
The function $f_{\alpha}$ is mollified by the convolution with the functions $\psi_{\delta}$ as
$$
f_{\alpha,\delta}(x) := (f_{\alpha}\ast \psi_{\delta})(x) =\int_{y\in\mathbb{R}^n} f_{\alpha}(y) \psi_{\delta}(x-y) dv(y).
$$
Notice ${\rm supp}\, f_{\alpha,\delta} \subset \{ x+y\,|\,x \in {\rm supp} f, y\in B_{\delta}(0)\}$ which is contained in $U_{\alpha}$ for a sufficiently small $\delta>0$.
The function $f$ on $M$ is now mollified by $\psi_{\delta}$ as $f_{\delta}(x) = \sum_{\alpha\in A} f_{\alpha,\delta}(x)$, $x\in M$.
It is shown that $f_{\delta}\in C^{\infty}(M)$ for a sufficiently small $\delta > 0$ and $\|f_{\delta}- f\|_{C^0} \rightarrow 0$ as $\delta\rightarrow 0$.

From above argument it is shown that the space $\mathcal{P}^{\infty}(M)$ is dense in $\mathcal{P}(M)$.
\end{proof}

Refer to \cite[4.2]{Pistone-2},\cite[4.4]{Brezis} and \cite[3.46]{Aubin} for the mollifieres on the euclidean space. 

Let $\{f_t\}$ be a family of continuous functions on $M$ parametrized in $t\in I$
($I$ is a closed interval) with $\dfrac{d f_t}{d t}\in C^0(M)$.
Then $\{f_t\}$ is mollified by $\psi_{\delta}$ as  a family of smooth functions $\{f_{t,\delta}\}$ and hence $\left\{\dfrac{d f_t}{dt}\right\}$ is mollified by $\left\{\dfrac{d f_{t,\delta}}{dt}\right\}$ so that $\left\|\dfrac{d f_{t,\delta}}{d t}- \dfrac{d f_t}{d t}\right\|_{C^0} \rightarrow 0$ as $\delta\rightarrow 0$.

\begin{prop}
The Riemannian distance in $\mathcal{P}(M)$ of $\mu,\mu_1\in\mathcal{P}^{\infty}(M)$ with respect to the metric $G$ is given by the Riemannian distance in $\mathcal{P}^{\infty}(M)$.
\end{prop}

\begin{proof}
Let $\mu,\mu_1$ be probability measures in $\mathcal{P}^{\infty}(M)$.
Then by definition the Riemannian distance $d(\mu,\mu_1)$ in $\mathcal{P}(M)$ is given by
\begin{equation}
d(\mu,\mu_1) = \inf_{c\in{\mathcal{C}}(\mu,\mu_1)} {\mathcal{L}}(c),
\end{equation}
where ${\mathcal{C}}(\mu,\mu_1)$ denotes the set of all piecewise $C^1$-curves $c : [0,1] \rightarrow \mathcal{P}(M)$, $c(0)=\mu$, $c(1)=\mu_1$.
To show the proposition we assume $\inf_{c\in{\mathcal{C}}(\mu,\mu_1)}{\mathcal{L}}(c) < \ell(\mu,\mu_1)$.
We will see by the aid of the mollifier argument  in the following that there exists a piecewise $C^1$-curve $c'$ which belongs to $\mathcal{P}^{\infty}(M)$ and satisfies ${\mathcal{L}}(c') < \ell(c'(0),c'(1))$.
This causes a contradiction, since $\ell(\cdot,\cdot)$ is the Riemannian distance function in $\mathcal{P}^{\infty}(M)$, as shown in Theorem \ref{smoothdistancetheorem}.

Set $\varepsilon = \frac{1}{2}\left(\ell(\mu,\mu_1)- \inf_{c\in{\mathcal{C}}(\mu,\mu_1}{\mathcal{L}}(c)\right)$.
Then, $\varepsilon > 0$ and there exists a piecewise $C^1$-curve $c$ in $\mathcal{P}(M)$ joining $\mu$ and $\mu_1$ and satisfying ${\mathcal{L}}(c) < \ell(\mu,\mu_1)- \varepsilon$. 

Write this curve $c$ as $c(t)= \mu_t = p(x,t)\lambda$ with $c(0)=\mu$ and $c(1)=\mu_1$, represented by $p(x)\lambda$ and $p_1(x)\lambda$, respectively, so $p(x,0)=p(x)$ and $p(x,1)=p_1(x)$.
By the above mollifier argument $p(x,t)$ and $\partial p(x,t)/\partial t$ are mollified by $p_{\delta}(x,t)$ and $\partial p_{\delta}(x,t)/\partial t$ so that as $\delta\rightarrow 0$ 
\begin{eqnarray*}
G_{\mu_{t,\delta}}\left(\frac{\partial \mu_{t,\delta}}{\partial t}, \frac{\partial \mu_{t,\delta}}{\partial t}\right) \rightarrow G_{\mu_t}\left(\frac{\partial \mu_t}{\partial t}, \frac{\partial \mu_t}{\partial t}\right).
\end{eqnarray*}
Here $\mu_{t,\delta} = p_{\delta}(x,t)\lambda$ gives us a piecewise $C^1$-curve $c_{\delta}$ joining $\mu_{\delta} = p_{\delta}(x,0)\lambda$ and $\mu_{1,\delta} = p_{\delta}(x,1)\lambda$ which both belong to $\mathcal{P}^{\infty}(M)$.
Thus we have
$$
\left\vert{\mathcal{L}}(c_{\delta}) - {\mathcal{L}}(c)\right\vert \leq \int_0^1 \left\vert \left\vert\dfrac{d\mu_{t,\delta}}{dt}\right\vert_{G} - \left\vert\dfrac{d\mu_t}{dt}\right\vert_G\right\vert dt < \dfrac{\varepsilon}{3}
$$
and consequently for sufficiently small $\delta>0$
\begin{equation}\label{lesserstimate}
{\mathcal{L}}(c_{\delta}) < \ell(\mu,\mu_1)- \frac{2 \varepsilon}{3},
\end{equation}
since ${\mathcal{L}}(c_{\delta}) < {\mathcal{L}}(c)+ \varepsilon/3 < \ell(\mu,\mu_1) - \varepsilon + \varepsilon/3$.

On the other hand, $\ell(\mu_{\delta}, \mu_{1,\delta})$,
the value of the function $\ell$ at $\mu_{\delta}$ and $\mu_{1,\delta}$ is the Riemannian distance in $\mathcal{P}^{\infty}(M)$ of $\mu_{\delta}$ and $\mu_{1,\delta}$. 
We find from the following that there exists $\delta_0 >0$ such that $\ell(\mu, \mu_{1})- \varepsilon/3 <\ell(\mu_{\delta}, \mu_{1,\delta})$ holds for any $0<\delta< \delta_0$.
In fact, we may assume $\ell(\mu_{\delta},\mu_{1,\delta}) \leq \ell(\mu,\mu_1)$.
Then 
\begin{equation}\label{differenceell}
\vert \ell(\mu_{\delta}, \mu_{1,\delta})- \ell(\mu, \mu_{1})\vert \leq \pi \left(\sin \frac{\ell(\mu,\mu_1)}{4}\right)^{-1}\left( \|p_{\delta}-p\|_{C^0}^{1/2}+\|p_{1,\delta}-p_1\|_{C^0}^{1/2}\right).
\end{equation}
By \eqref{cosellfunction} in section \ref{continuity} we have
\begin{equation}
\left\vert\cos \frac{\ell(\mu_{\delta},\mu_{1,\delta})}{2}- \cos \frac{\ell(\mu,\mu_1)}{2}\right\vert \leq \left(\int_M\vert\sqrt{p_{\delta}}-\sqrt{p}\vert^2\,d \lambda\right)^{1/2} + \left(\int_M\vert\sqrt{p_{1,\delta}}-\sqrt{p_1}\vert^2\,d \lambda\right)^{1/2}
\end{equation}
to which we apply the inequality $\vert\sqrt{a}-\sqrt{b}\vert^2 \leq \vert a-b\vert$ for $a, b \geq 0$ to get
\begin{align*}
\left\vert\cos \frac{\ell(\mu_{\delta},\mu_{1,\delta})}{2}- \cos \frac{\ell(\mu,\mu_1)}{2}\right\vert 
\leq& \left(\int_M \vert p_{1,\delta}(x)- p_1(x)\vert\,d\lambda\right)^{1/2} + \left(\int_M \vert p_{\delta}(x)-p(x)\vert\,d\lambda\right)^{1/2} \\
\leq& \|p_{1,\delta}-p_1\|_{C^0}^{1/2} + \|p_{\delta}-p \|_{C^0}^{1/2}.
\end{align*}
Therefore, by setting $L= \ell(\mu,\mu_1)$, $L_{\delta}= \ell(\mu_{\delta},\mu_{1,\delta})$ for simplicity one has
\begin{equation*}
2 \left\vert\sin \frac{L+L_{\delta}}{4} \sin \frac{L-L_{\delta}}{4}\right\vert =\left\vert\cos \frac{\ell(\mu_{\delta},\mu_{1,\delta})}{2}- \cos \frac{\ell(\mu,\mu_1)}{2}\right\vert
\end{equation*}and  since $L, L_{\delta}\in (0,\pi)$ and $L_{\delta}\leq L$ from the assumption, one sees $(L +L_{\delta})/4\geq L/4$ and $(L-L_{\delta})/4 \leq \pi/2$. Since $(2/\pi)\cdot x\leq \sin x$, $x\in[0,\pi/2]$, one obtains 
\begin{equation*}
2  \cdot \frac{2}{\pi}\cdot \frac{(L-L_{\delta})}{4} \sin \frac{L}{4}  \leq 
2 \left\vert\sin \frac{L+L_{\delta}}{4} \sin \frac{L-L_{\delta}}{4}\right\vert
\end{equation*}
and hence
\begin{equation}
\frac{(L-L_{\delta})}{\pi} \sin \frac{L}{4}  \leq \|p_{1,\delta}-p_1\|_{C^0}^{1/2} + \|p_{\delta}-p \|_{C^0}^{1/2}
\end{equation}
from which the desired inequality \eqref{differenceell} is obtained. 

Now, $p$ and $p_1$ have been mollified as above by $p_{\delta}$, $p_{1,\delta}$, respectively so,  by the aid of  \eqref{differenceell},
we can take  $\delta_1>0$ such that $\vert L_{\delta}-L\vert=\vert\ell(\mu_{\delta},\mu_{1,\delta})-\ell(\mu,\mu_1)\vert < \varepsilon/3$ holds for any $\delta$ satisfying $0<\delta<\delta_1$, so we have $\ell(\mu,\mu_1) -\varepsilon/3 < \ell(\mu_{\delta},\mu_{1,\delta})$. Therefore,  from \eqref{lesserstimate} it follows that for sufficiently small $\delta$ the length of $c_{\delta}$ satisfies ${\mathcal{L}}(c_{\delta}) < \ell(\mu,\mu_1)- 2\varepsilon/3 < \ell(\mu,\mu_1)- \varepsilon/3 < \ell(\mu_{\delta},\mu_{1,\delta})$.
This leads a contradiction, since $\ell(\mu_{\delta},\mu_{1,\delta})$ is distance of $\mu_{\delta}$ and $\mu_{1,\delta}$ in $\mathcal{P}^{\infty}(M)$.
Thus, we can assert that the function $\ell$ gives the Riemannian distance of two measures $\mu,\mu_1$ of $\mathcal{P}^{\infty}(M)$ not only in $\mathcal{P}^{\infty}(M)$ but also in $\mathcal{P}(M)$. 
\end{proof} 

\begin{proof}[P\,r\,o\,o\,f of Theorem \ref{distancetheorem}.]
The Riemannian distance in $\mathcal{P}(M)$ of probability measures $\mu$ and $\mu_1$ which belong to $\mathcal{P}(M)$ is given by $\inf_{c\in{\mathcal{C}}(\mu,\mu_1)} {\mathcal{L}}(c)$. We assume $\inf_c {\mathcal{L}}(c) < \ell(\mu,\mu_1)
$.
Then, the proof of Theorem \ref{smoothdistancetheorem} is also applied, even though $\mu$, $\mu_1$ admit a continuous density function, but by a minor modification.
From the arguments at the proof of Theorem  \ref{smoothdistancetheorem}, we obtain $\inf_c {\mathcal{L}}(c) = \ell(\mu,\mu_1)$ which implies that $\ell(\cdot,\cdot)$ gives the Riemannian distance in $\mathcal{P}(M)$ with respect to the Fisher metric $G$.
\end{proof}

\section{The topology and the smooth structure of $\mathcal{P}(M)$}\label{relevanttopology}

\subsection{Affine structure and local coordinate maps}
In this section we introduce certain topology and a smooth structure on $\mathcal{P}(M)$ by means of the argument of Pistone and Sempi developed in \cite{PistoneS}.
For this purpose, let $(\Omega, {\mathcal B}, \lambda)$ be a probability space in a more general setting and denote by ${\mathcal M}_{\lambda}$ the set of $L_1$-integrable density functions of all the probability measures $\mu$ equivalent to $\lambda$, i.e., $\mu \ll \lambda$, $\lambda \ll \mu$,
\begin{equation}
{\mathcal M}_{\lambda} := \left\{\mu\ \left\vert\ \frac{d\mu}{d\lambda}(=p)\in L_1(\Omega,\lambda),\, \mbox{$p > 0$ $\lambda$-a.s.}, E_{\lambda}\left[\frac{d\mu}{d\lambda}\right] = 1 \right.\right\}.
 \end{equation}
$E_{\lambda}[\,\cdot\,]$ is the expectation with respect to $\lambda$.
Let $\mu = p \lambda$ be an arbitrary probability measure of ${\mathcal M}_{\lambda}$.
For a real valued random variable $u$, i.e., a measurable function on $(\Omega, {\mathcal B}, \mu)$ we denote by ${\hat u}_{\mu}(t)$ the moment generating function of $u$, defined by $\displaystyle {\hat u}_{\mu}(t):= \int_{\Omega} \exp(t u) d\mu = E_{\mu}[\exp(tu)]$.
Define for each $\mu$ a vector space consisting of certain random variables:
\begin{equation}
V_{\mu}
:=\left\{ u\in L_1(\Omega, \mu)\, \left|\, 0\in D({\hat u}_{\mu})^0,\, E_{\mu}[u] = 0\right.\right\}.
\end{equation}
The first condition, $0\in D({\hat u}_{\mu})^0$ means that the domain of ${\hat u}_{\mu}$ contains a neighborhood of $0$ in $\mathbb{R}$.
Then, $V_{\mu}$ turns out to be a closed linear subspace of a Banach space $L^{\phi}(\Omega, \mu)$, the Orlicz space of the Young function $\phi=\phi(t)$:
\begin{equation}
L^{\phi}(\Omega, \mu) := \left\{\mbox{$u$ is a random variable}\, \left\vert\, \exists\, a > 0, E_{\mu}\left[\phi\left(\frac{u}{a}\right)\right] < +\infty
\right.\right\}
\end{equation}
with norm 
\begin{equation}
\| u\vert\|_{\phi,\mu}
:= \inf \left\{a > 0\, \left\vert\, E_{\mu}\left[\phi\left(\frac{u}{a}\right)\right] \leq 1 \right.\right\}.
\end{equation}
Note the Young function $\phi(t)$ is a real valued convex, even function on $\mathbb{R}$ satisfying $\phi(0)=0$ and strictly increasing for $t>0$ with $\displaystyle\lim_{t\rightarrow\infty} t^{-1}\phi(t)= +\infty$.
In \cite{PistoneS}\, $\phi(t) = \cosh t - 1$ is especially adopted.
The Orlicz space $L^{\phi}(\Omega, \mu)$ of the Young function 
$\phi$ is the generalization of the space $L_p(\Omega,\mu)$ of $L_p$-integrable functions on $\Omega$, $p\geq 1$. For a precise argument refer to \cite{PistoneS}.
It is shown in \cite{PistoneS} that $V_{\mu}$ coincides with the closed linear subspace 
\begin{equation*}
L^{(\cosh\, -\, 1)}_0(\Omega, \mu)= \{ u \in L^{(\cosh-1)}(\Omega, \mu)\, \vert\, E_{\mu}[u]=0\}\subset L^{(\cosh-1)}(\Omega, \mu)
\end{equation*}
and the following holds:
\begin{equation}
L_{\infty,0}(\Omega,\mu) \hookrightarrow V_{\mu} (= L^{(\cosh-1)}_0(\Omega, \mu)) \hookrightarrow \bigcap_{p>1} L_{p,0}(\Omega,\mu),
\end{equation}
where the symbol ``$\hookrightarrow$'' means a continuous and dense embedding.
The space $\mathcal{P}(M)$, our main subject in this paper, turns out to be a dense subset of ${\mathcal M}_{\lambda}$ for $(\Omega=M, {\mathcal B}=\mathcal{B}(M), \lambda)$.

Let ${\mathcal V}_{\mu} = \{ u \in L^{(\cosh - 1)}(\Omega, \mu)\, \vert\, \vert\vert u \vert\vert_{\phi,\mu} < 1\} \cap V_{\mu}$ be a unit open ball in $V_{\mu}$. Then, the injective map
\begin{equation}\label{sigmamap}
\sigma_{\mu} : {\mathcal V}_{\mu} \ni u \mapsto \exp\left[ u - \Psi_{\mu}(u)\right] \mu = \frac{\exp u}{E_{\mu}[\exp u]} \mu \in {\mathcal M}_{\lambda}
\end{equation}
together with ${\mathcal U}_{\mu}= \sigma_{\mu}({\mathcal V}_{\mu})$, the image of ${\mathcal V}_{\mu}$ and $s_{\mu}=\sigma_{\mu}^{-1}$, the inverse map of $\sigma_{\mu}$, yields a chart of ${\mathcal M}_{\lambda}$ around $\mu$.
Here $\Psi_{\mu}(u) = \log E_{\mu}[\exp u]$ is the cumulant generating function of $u$. Notice that $s_{\mu}$ has the form
\begin{equation}\label{smap}
s_{\mu}(\nu) = \log \left(\frac{d\nu}{d\mu}\right) - E_{\mu}\left[\log \left(\frac{d\nu}{d\mu}\right)\right],\quad\nu\in {\mathcal U}_{\mu}
\end{equation}
so that the transition function between $s_{\mu}\left({\mathcal U}_{\mu} \cap {\mathcal U}_{\mu_1} \right)$ and $s_{\mu_1}\left({\mathcal U}_{\mu} \cap {\mathcal U}_{\mu_1} \right)$ of ${\mathcal M}_{\lambda}$ is represented by an affine transform of the form
\begin{equation*}
s_{\mu_1}\circ s_{\mu}^{-1}(u) = u + \log \left(\frac{d\mu}{d\mu_1}\right)
 - E_{\mu_1}\left[u + \log \left(\frac{d\mu}{d\mu_1}\right)
\right].
\end{equation*}
\begin{thm}[{\cite[Theorem 3.3]{PistoneS}}]\label{affinestr}
The collection of pairs $\displaystyle\left\{\left({\mathcal U}_{\mu}, s_{\mu}\right)\,|\,\mu\in {\mathcal M}_{\lambda}\right\}$ defines an affine smooth atlas on ${\mathcal M}_{\lambda}$. 
\end{thm}
The atlas of ${\mathcal M}_{\lambda}$ necessarily induces a topology which is shown to be equivalent to the topology induced from the $e$-convergence defined in \cite[definition 1.1]{PistoneS}.

From this theorem the map $\varphi$, the normalized geometric mean, given in Definition \ref{normalizedgeometricmean} turns out to be smooth.
In fact, one can represent $\varphi$ as the arithmetic mean in terms of the local coordinate maps $\sigma_{\mu}$ and $s_{\mu}$.

\begin{lem}
\begin{eqnarray}\label{localrepresentation}
s_{\mu_1}\left(\varphi(\sigma_{\mu}(u),\sigma_{\mu'}(u')\right) = \frac{1}{2}\left\{u-E_{\mu}[u] + u'- E_{\mu'}[u'] \right\}, \hspace{2mm}u\in {\mathcal{V}}_{\mu}, u'\in {\mathcal{V}}_{\mu'},
\end{eqnarray}where one sets $\mu_1=\varphi(\mu,\mu')$ for  $\mu,\mu'\in{\mathcal{M}}_{\lambda}$.
\end{lem}

\begin{proof}
 This is given by a slight computation from the formula
\begin{equation}\label{varphilocal}
\varphi\left(\sigma_{\mu}(u),\sigma_{\mu'}(u')\right)= \frac{1}{\int_M\sqrt{\exp u\exp u'} d\mu_1}\, \sqrt{\exp u\exp u'}\hspace{0.5mm} \mu_1
\end{equation}
together with \eqref{smap}.

\eqref{varphilocal} is derived as follows. 
From Definition \ref{normalizedgeometricmean}
\begin{equation}
\varphi\left(\sigma_{\mu}(u),\sigma_{\mu'}(u')\right)=\frac{1}{\int_M\sqrt{\frac{d\sigma_{\mu'}(u')}{d\sigma_{\mu}(u)}} d\sigma_{\mu}(u)} \sqrt{\frac{d\sigma_{\mu'}(u')}{d\sigma_{\mu}(u)}}\,\sigma_{\mu}(u)
\end{equation}
where
\begin{equation}
\frac{d\sigma_{\mu'}(u')}{d\sigma_{\mu}(u)}
= \frac{\exp u' (E_{\mu'}[\exp u'])^{-1}}{\exp u (E_{\mu}[\exp u])^{-1}} \frac{d\mu'}{d\mu}
\end{equation}
which is ensured by the Radon-Nikodym derivative of the measures $\sigma_{\mu}(u)$, $\sigma_{\mu'}(u')$ with respect to the measures $\mu$, $\mu'$, respectively.
Then
\begin{equation}
\sqrt{\frac{d\sigma_{\mu'}(u')}{d\sigma_{\mu}(u)} }\,\sigma_{\mu}(u)
= \sqrt{\frac{\exp u\exp u'}{E_{\mu}[\exp u] E_{\mu'}[\exp u']}} \sqrt{\frac{d\mu'}{d\mu}}\,\mu.
\end{equation}
Since $\displaystyle\mu_1= \varphi(\mu,\mu')$, one finds $\displaystyle\sqrt{\frac{d\mu'}{d\mu}} \mu= \left(\int_M\sqrt{\frac{d\mu'}{d\mu}}d\mu\right) \mu_1$ so 
\begin{equation}
 \int_M \sqrt{\frac{d\sigma_{\mu'}(u')}{d\sigma_{\mu}(u)} }\,d\sigma_{\mu}(u)= \frac{\int \sqrt{\frac{ d\mu'}{d\mu}}d\mu}{\sqrt{E_{\mu}[\exp u]\ E_{\mu'}[\exp u']}}\, \int_M \sqrt{\exp u \exp u'}\,d\mu_1
\end{equation}
and
\begin{equation}
 \sqrt{\frac{d\sigma_{\mu'}(u')}{d\sigma_{\mu}(u)} } \sigma_{\mu}(u)= \frac{\int \sqrt{\frac{ d\mu'}{d\mu}}d\mu}{\sqrt{E_{\mu}[\exp u]\ E_{\mu'}[\exp u']}}\,\sqrt{\exp u \exp u'} \mu_1
\end{equation}
from which \eqref{varphilocal} follows.
\end{proof}

The smoothness of $\varphi$ at any $(\mu,\mu')$ is immediately derived from \eqref{localrepresentation}.

As the space ${\mathcal M}_{\lambda}$ can be treated as an affine manifold, in the rest of this section by applying the definition of tangent vectors to ${\mathcal M}_{\lambda}$ given in \cite{PistoneS}, we present the Fisher information metric in local coordinate expression. 
Now let $c : I \rightarrow {\mathcal M}_{\lambda}$ be a $C^1$--curve of ${\mathcal M}_{\lambda}$ with $c(t_0)\in {\mathcal U}_{\nu}$ with respect to a chart $({\mathcal U}_{\nu},s_{\nu})$ associated to $\nu\in {\mathcal{M}}_{\lambda}$, where $I$ is an open interval.
We have then the $C^1$--curve $u_{\nu}(t)=s_{\nu}\circ c(t)$ in ${\mathcal V}_{\nu}$ in the form of
\begin{equation}
u_{\nu}(t) = \log \left(\frac{dc(t)}{d\nu} \right)- E_{\nu}\left[\log \left(\frac{dc(t)}{d\nu} \right)\right]
\end{equation}
with velocity vector $u'_{\nu}(t_0) = \left(d s_{\nu}\right)_{c(t_0)}\left(c'(t_0)\right)$ belonging to $V_{\nu}$;
\begin{equation}
u'_{\nu}(t_0) = \Big\{\frac{d}{dt} \log \left(\frac{dc(t)}{d\nu} \right)- \frac{d}{dt} E_{\nu}\left[\log \left(\frac{dc(t)}{d\nu} \right)\right]\Big\}_{\vert_{t_0}}. 
\end{equation} 
When $c(t_0)\in {\mathcal U}_{\nu_1}$, with respect to another chart $({\mathcal U}_{\nu_1},s_{\nu_1})$, we have similarly the $C^1$--curve $u_{\nu_1}(t)= s_{\nu_1}\circ c(t)$ in ${\mathcal V}_{\nu_1}$ with velocity vector $u'_{\nu_1}(t_0) \in V_{\nu_1}$.
Therefore, it is shown from the affine structure of the space $\mathcal{M}_{\lambda}$, stated in Theorem \ref{affinestr} that the difference $u'_{\nu_1}(t_0)- u'_{\nu}(t_0)$ is a constant function and from this fact the tangent vector of $c(t)$ at $t=t_0$ in local coordinate expression is defined as the collection of such velocity vectors and denote it by $[c'(t_0)]$.
The set of all tangent vectors is a vector space, denoted by $T_{c(t_0)}\mathcal{M}_{\lambda}$.
To formulate Fisher information metric in local coordinate expression we select a velocity vector which is {\it particular} from the collection $[c'(t_0)]$, $u'_{\mu}(t_0)= (ds_{\mu})_{\mu}(c'(t_0))$, where $u_{\mu}(t)= s_{\mu}\circ c(t)$ is a curve in ${\mathcal V}_{\mu}$ with respect to a chart $({\mathcal U}_{\mu},s_{\mu})$ for which $c(t_0)=\mu$.
Notice that $u_{\mu}(t_0)=0$ and $u_{\mu}(t)\in V_{\mu}$ for any $t$ and hence $u'_{\mu}(t_0) \in {\rm Ker} E_{\mu}$.
By using particular tangent vectors we have
\begin{defn}
Let $\tau$, $\tau_1 \in T_{\mu} {\mathcal M}_{\lambda}$ be tangent vectors at $\mu$, and $[u]$, $[u_1]$ be the corresponding tangent vectors in local coordinate expression, respectively.
Then the scalar product of $[u], [u_1]$ is defined by $\displaystyle\langle [u], [u_1]\rangle_{\mu} = \int_{\Omega} u'_{\mu}(t_0)\, u'_{1\mu}(t_0)\,d\mu$, where $u'_{\mu}(t_0)$ and $u'_{1\mu}(t_0)$ are particular velocity vectors of the curves $u_{\mu}(t)= s_{\mu}\circ c(t)$, $u_{1\mu}(t)= s_{\mu}\circ c_1(t)$ representing $[u]$, $[u_1]$, respectively, where $c(t):= \mu+(t-t_0) \tau$, $c_1(t):= \mu+(t-t_0) \tau_1$ are the corresponding curves in ${\mathcal M}_{\lambda}$. 
\end{defn}

\begin{note}
The scalar product is stated in \cite{PistoneS} as a quadratic form on $T_{\mu}{\mathcal M}_{\lambda}$.
It can be represented in the following form:
\begin{align}\label{covariance}
\langle [u], [u_1]\rangle_{\mu} 
=& \int_{\Omega} \left(u'_{\nu}(t_0)- E_{\mu}[u'_{\nu}(t_0)]\right)
\left(u'_{1\nu}(t_0)- E_{\mu}[u'_{1\nu}(t_0)]\right)\,d\mu \\
=&\label{covarianceformula} \int u'_{\nu}(t_0)\cdot u'_{1\nu}(t_0)\,d \mu - E_{\mu}[u'_{\nu}(t_0)]\cdot E_{\mu}[u'_{1\nu}(t_0)], 
\end{align}
where $u'_{\nu}(t_0)$, $u'_{1\nu}(t_0)$ are the vectors representing $[u]$, $[u_1]$ with respect to other chart $({\mathcal U}_{\nu},s_{\nu})$, respectively.
The formula \eqref{covariance} is viewed as the covariance \eqref{covarianceformula} of two random variables.
\eqref{covariance} stems from the fact that the difference of the vector $u'_{\nu}(t_0)$ and the particular one $u'_{\mu}(t_0)$ is $u'_{\nu}(t_0)- u'_{\mu}(t_0) = E_{\mu}[u'_{\nu}(t_0)]$.
It is indicated in \cite[3.4]{PistoneS} that the cumulant 2-form has a representation of the scalar product (covariance).
\end{note}

\begin{prop}\label{localchartform}
The scalar product, thus defined, coincides with Fisher information metric $G$, namely,
\begin{equation}\label{fishermetriccovariance}
\langle [u], [u_1]\rangle_{\mu} = G_{\mu}(\tau,\tau_1),\quad\tau,\tau_1\in T_{\mu}{\mathcal M}_{\lambda}.
\end{equation}
Here $[u], [u_1]$ are the corresponding tangent vectors of $\tau, \tau_1$, respectively in local coordinate expression.
\end{prop}

In fact, the left hand side of \eqref{fishermetriccovariance} has the form of $\displaystyle\int \frac{q(x)}{p(x)}\, \frac{q_1(x)}{p(x)}\,p(x)\,d\lambda(x)$ where $p$, $q$ and $q_1$ are the density functions of $\mu$, $\tau$ and $\tau_1$ with respect to $\lambda$, respectively.
Since $p + (t-t_0) q$ is the density function of the curve $c(t)$,
one finds
$$
u_{\mu}(t) = \log \frac{p+(t-t_0) q}{p} - E_{\mu}\left[\log \frac{p+(t-t_0) q}{p} \right]$$
and hence
$$
u'_{\mu}(t_0)
=\left.\frac{d}{dt}\right\vert_{t=t_0}
\left(\log \frac{p+(t-t_0) q}{p}- E_{\mu}\left[\log \frac{p+(t-t_0) q}{p}\right]\right)
= \frac{q}{p}.
$$
Similarly one has $\displaystyle u'_{1\mu}(t_0) = \frac{q_1}{p}$ to obtain \eqref{fishermetriccovariance}.

\subsection{Connectedness by open mixture arc and constant vector fields}
\label{openmixturearc}

We close this section by giving a certain comment on a constant vector field.
By using constant vector fields, Friedrich obtains in \cite{F1991} the formulae of Levi-Civita connection and geodesics without any argument of the coordinate structure of the space of probability measures.

By using the notion of being connected by an open mixture arc introduced in \cite{CPistone} (see also \cite{Santacrocexx}), the argument of constant vector fields is well treated.
Two probability measures $\mu=p\lambda$, $\mu_1= p_1\lambda$ of $\mathcal{P}(M)$ are connected by an open mixture arc if there exists an open interval $I (\supset [0,1])$ such that $t \mu+ (1-t)\mu_1$ belongs to $\mathcal{P}(M)$ for every $t\in I$.
Here we denote by  $\mathcal{P}(M)$ the space of probability measures $\mu= p\lambda$ which satisfy $\mu \ll \lambda$ with $p\in C^0_+(M)$, where $\lambda$ is the Riemannian volume form on a complete Riemannian manifold $M$ of unit volume. $M$ is not necessarily assumed to be compact.
We easily find that this notion is an equivalence relation from \cite[Theorem 4.11]{Santacrocexx}. 
Moreover this theorem asserts that $\mu= p\lambda$ and $\mu_1=p_1\lambda$ are connected by an open mixture arc if and only if there exist constants $c_1, c_2$ with $0 < c_1 < 1< c_2$ such that $c_1 < d\mu_1/d\mu (x) (=p_1(x)/p(x)) < c_2$   for any $x\in M$. Therefore, letting $\mathcal{P}_m(M)$ be the space of probability measures $\mu=p\lambda\in \mathcal{P}(M)$ which are connected with $\lambda$ by an open mixture arc.  Notice that  arbitrary $\mu,\mu_1$  belonging to $\mathcal{P}_m(M)$ are connected by an open mixture arc each other. $\mathcal{P}_m(M)$ coincides with $\mathcal{P}(M)$, provided $M$ is compact. We define  a constant vector field at every probability measure of $\mathcal{P}_m(M)$ as follows.

\begin{prop}\label{definitiontangentspace}
Set
\begin{equation}
{\bf{V}}_m(M) :=\left\{ \nu = q \lambda\, \left\vert\, q\in C^0(M), \int_M d\nu = 0,\, \lambda+t \nu \in \mathcal{P}_m(M)\ \mbox{\rm for any}\ t\in (-\varepsilon,\varepsilon)\right.\right\},
\end{equation}
regarded as the tangent space at $\lambda$ to $\mathcal{P}_m(M)$; $T_{\lambda}\mathcal{P}_m(M)$.
Here $\varepsilon >0$ is a constant which may depend on $\nu$.
Then,
\begin{enumerate}
\item ${\bf V}_m(M)$ is a vector space. 
\item Every $\tau\in {\rm{\bf{V}}}_m(M)$ induces a constant vector field at every $\mu\in \mathcal{P}_m(M)$. In other words, each $\tau\in {\bf V}_m(M)$ yields measures $\mu+ t\tau$ in  $\mathcal{P}_m(M)$, $t\in(-\varepsilon,\varepsilon)$ for any $\mu\in\mathcal{P}_m(M)$.
\end{enumerate}
\end{prop}
  
\begin{proof}
First we show that ${\bf V}_m(M)$ is a vector space. Let $\tau=q\lambda$ and  $\tau'=q'\lambda\in {\bf V}_m(M)$. 
From the positivity of density function of $\lambda + t \tau$ there exists $\varepsilon >0$ such that $1+ tq(x) > 0$ for any $t\in(-\varepsilon,\varepsilon)$.  
Moreover, from connectedness by an open mixture arc one asserts that from \cite[Theorem 4.11]{Santacrocexx} for any fixed $t\in(-\varepsilon,\varepsilon)$ there exist constants $0<k_1<1<k_2$ such that
\begin{eqnarray}\label{tau}
0 < k_1 < \frac{d(\lambda+ t\tau)}{d\lambda} (x) = 1+ tq(x) < k_2, \quad x\in M,
\end{eqnarray}  
This indicates  aside the boundedness of $\vert q\vert$, as $\vert q(x)\vert < 2/\varepsilon\, \max\{k_2-1, 1-k_1\}$, $x\in M$ by letting $t=\varepsilon/2$.
  
It is easily seen that $c \tau\in {\bf V}_m(M)$ for any $c\in\mathbb{R}$.
We see next that  $\tau + \tau'$ belongs to ${\bf V}_m(M)$ as follows.
For $\tau'$ we have similarly as $\tau$ that for any fixed $t\in(-\varepsilon',\varepsilon')$ there exist constants $0<k'_1<1<k'_2$ such that
\begin{eqnarray}\label{taudash}
0 < k'_1 < \frac{d(\lambda+ t\tau')}{d\lambda} (x) = 1+ tq(x) < k'_2, \quad x\in M.
\end{eqnarray}
Then, from \eqref{tau}, \eqref{taudash} we have
\begin{align*}
 \frac{1}{2}(k_1+k_1')
 <&  \frac{1}{2}(1+2tq(x)+1+ 2tq'(x)) = 1+ t(q(x)+q'(x)) \\
 <& (1+ t(q(x))+ (1+tq'(x)) < k_2+ k_2'
\end{align*}
for any $t$ satisfying $-1/2\min\{\varepsilon,\varepsilon'\} < t< 1/2 \min\{\varepsilon,\varepsilon'\}.$ Hence, this shows that $\tau+\tau'$ belongs to ${\bf V}_m(M)$.

(ii) is shown as follows.
Let $\mu=p\lambda\in \mathcal{P}_m(M)$ and $\tau=q\lambda\in {\bf V}_m(M)$  be arbitrary.  Since $\mu$ is connected with $\lambda$ by an open mixture arc, there exist constants $0 < c_1 < 1< c_2$ such that $c_1 < d\mu/d\lambda(x) = p(x)< c_2$, $x\in M$ and thus \, $c_1 + t q(x) < p(x)+ t q(x) < c_2 + t q(x)$, $x\in M$.
Hence 
\begin{equation*}
c_1\left(1 + \frac{t}{c_1}\cdot q(x)\right) < p(x)+ t q(x) < c_2\left(1 + \frac{t}{c_2}\cdot q(x)\right),\quad x\in M.
\end{equation*}
We may assume \eqref{tau} for this $\tau$.
Then, for any fixed $t$ satisfying $-\varepsilon c_1 < t < \varepsilon c_1$ one has $p(x)+ t q(x) > c_1(1+ tq(x)/c_1) > c_1 k_1> 0$ for all $x\in M$ and similarly $p(x)+tq(x) < c_2(1+ tq(x)/c_2) < c_2k_2$.
These imply that $\mu+ t \tau$, $- c_1\varepsilon < t < c_1\varepsilon$ defines a probability measure in $\mathcal{P}_m(M)$, namely $\tau$ induces a tangent vector at $\mu$ and hence a constant vector field everywhere on $\mathcal{P}_m(M)$.
\end{proof}

For any $\mu$, $\mu_1$ of $\mathcal{P}_m(M)$ their difference $\mu_1-\mu$ belongs to ${\bf V}_m(M)$.

From this proposition the inner product $G_{\mu}(\tau,\tau')$  for $\tau,\tau'\in{\bf V}_m(M)$, $\mu\in\mathcal{P}_m(M)$ is well defined, since $1/p(x)$ and $\vert q(x)\vert, \vert q'(x)\vert$ are bounded from above.

\begin{rem}
By using the constant vector field technique employed by T. Friedrich in \cite{F1991} together with the notion of connectedness by an open mixture arc, we study geodesics on the space of probability measures directly, not via the local coordinate maps $\sigma_{\mu}, s_{\mu}$ defined in \cite{PistoneS}.
Gaussian measure $\mu_{(c, d)}$ of mean value $c$ and variance $d(>0)$ on the one-dimensional euclidean space $\mathbb{R}$ is connected with Gaussian measure $\mu_{(c_1, d_1)}$ if and only if $(c_1, d_1)=(c, d)$.
Therefore, for a space of probability measures on $\mathbb{R}$ including all Gaussian measures it is hard to use the notion of connectedness by open mixture arc so that the notion of open exponential arc together with the local coordinate maps $\sigma_{\mu}$, $s_{\mu}$ of \cite{PistoneS} seems to be applied. 
\end{rem}


\bigskip

\noindent\textbf{\Large Acknowledgement}\medskip

The authors would like to thank the referees for indicating the authors valuable comments and relevant references.

\end{document}